\documentclass[12pt]{article}
\usepackage[T1]{fontenc}
\usepackage{fouriernc}
\usepackage{amsmath,amssymb,amsthm}
\usepackage{array}
\usepackage{longtable}
\usepackage[paperwidth=595.276bp,paperheight=841.89bp,left=72bp,right=72bp,top=72bp,bottom=72bp,headsep=20bp,footskip=36bp]{geometry}
\usepackage{fancyhdr}
\usepackage{cite}
\usepackage[hidelinks]{hyperref}
\newcommand{\doi}[1]{\href{https://doi.org/#1}{doi:\detokenize{#1}}}
\newcommand{\arxiv}[1]{\href{https://arxiv.org/abs/#1}{arXiv:#1}}

\fancypagestyle{plain}{\fancyhf{}\fancyfoot[C]{\thepage}}
\makeatletter
\def\@dnfnsymbol#1{\ensuremath{\ifcase#1\or *\or \dagger\or \ddagger\or
  \mathsection\or \mathparagraph\or \|\or **\or \dagger\dagger\or \ddagger\ddagger\or
  \mathsection\mathsection\or \mathparagraph\mathparagraph\or ***\or
  \dagger\dagger\dagger\else\@ctrerr\fi}}
\renewcommand{\thefootnote}{\@dnfnsymbol\c@footnote}
\makeatother

\newcommand{\B}[1]{\mathbf{#1}}
\newcommand{\A}{\mathcal{A}}
\newcommand{\Var}{\mathcal{V}ar}
\newcommand{\Vz}{\mathcal{V}_0}
\newcommand{\Aphi}{\mathcal{A}_{\varphi}}
\newcommand{\cyc}{\circlearrowleft}
\newcommand{\lb}{\llbracket}
\newcommand{\rb}{\rrbracket}

\theoremstyle{plain}
\newtheorem{theorem}{Theorem}
\newtheorem{definition}[theorem]{Definition}
\newtheorem{proposition}[theorem]{Proposition}
\newtheorem{lemma}[theorem]{Lemma}
\newtheorem{corollary}[theorem]{Corollary}
\newtheorem{example}[theorem]{Example}
\newtheorem{remarkn}[theorem]{Remark}
\newtheorem*{remark}{Remark}

\begin{document}

\noindent{\Large $\delta$-Novikov and $\delta$-Novikov--Poisson algebras\footnote{We express our sincere gratitude to Farukh Mashurov for his meticulous editing of the initial draft
and for the numerous insightful discussions that significantly shaped the development of this work. We are also deeply grateful to the anonymous reviewer who prepared a report on this article for the Bulletin of the Iranian Mathematical Society (where it was accepted for publication but subsequently rejected four months later on the grounds that it was already available on arXiv.org) for their careful study of the manuscript and constructive comments. We also thank the anonymous reviewer of the present version for the example in Proposition \ref{prop43}.}\,\footnote{The present version of the paper was checked in the proof assistant Lean 4 (with the Mathlib library)
with the help of Aristotle (\url{https://aristotle.harmonic.fun/}).}\,\footnote{The work was supported by FCT 2023.08031.CEECIND and UID/00212/2025.}}

\vspace{1.5em}
\begin{center}
{\large\bfseries Ivan Kaygorodov}\footnote{CMA-UBI, University of Beira Interior, Covilh\~a, Portugal;\ \ kaygorodov.ivan@gmail.com}
\end{center}
\vspace{1em}

\noindent\textbf{Abstract:} \textit{This article considers the structure and properties of $\delta$-Novikov algebras, a
generalization of Novikov algebras characterized by a scalar parameter $\delta$. Unlike Novikov
algebras, they have simple non-commutative finite-dimensional algebras for $\delta=-1;$ on the
other hand, for every $\delta\neq1$, not all of them embed into commutative associative algebras
with a $\delta$-derivation. We introduce $\delta$-Novikov--Poisson algebras, extending several theorems
from classical Novikov--Poisson algebras. For $\delta\neq1$, the commutator algebras of
$\delta$-Novikov algebras are metabelian Lie algebras, and every identity of these commutator
algebras follows from the identities of metabelian Lie algebras. We
construct $\delta$-Poisson and transposed $\delta$-Poisson algebras through $\delta$-derivations on commutative
associative algebras, and show that for every $\delta$ there are transposed $(\delta+1)$-Poisson
algebras that cannot be embedded into $\delta$-Novikov--Poisson algebras under the commutator product. Finally, we prove that the $\delta$-Novikov
operad is not Koszul for any value of $\delta$, extending the known results for $\delta=1$ and $\delta=0$,
and that a minimal ideal $I$ of a $\delta$-Novikov algebra is simple or satisfies $I^2=0$.}



\tableofcontents

\newpage

\section*{Introduction}
The variety of Novikov algebras is defined by the identities
\[
(xy)z-x(yz)=(yx)z-y(xz),\qquad (xy)z=(xz)y.
\]
It is the intersection of the variety of right commutative algebras and the variety of left symmetric
(pre-Lie) algebras (see, for example, \cite{b2a}); it contains the commutative associative algebras, every
Novikov algebra is Lie-admissible, and Novikov algebras are closely related to Novikov--Poisson
algebras \cite{b46}. Novikov (also known as Gelfand--Dorfman--Novikov) algebras appeared in the papers
of Gelfand and Dorfman \cite{b27} and of Balinsky and Novikov \cite{b6}. Their systematic algebraic study
started after the paper of Zelmanov, where all complex finite-dimensional simple Novikov algebras
were classified \cite{b47}; infinite-dimensional simple Novikov algebras were constructed and studied by
Filippov, Osborn and Xu \cite{b24,b40,b45}. Many other purely algebraic properties of Novikov algebras
were studied in a series of papers by Dzhumadildaev and his co-authors (see, for example, \cite{b18,b19,b22}
and references therein). Dzhumadildaev and L\"ofwall described a basis of free Novikov algebras \cite{b22};
Sartayev and co-authors described a basis of free metabelian Novikov algebras \cite{b14}. Dzhumadildaev
proved that the Novikov operad is not Koszul \cite{b18}, and Dzhumadildaev and Ismailov found the
$S_n$-module structure of the multilinear component of degree $n$ of the free Novikov algebra \cite{b19}.
Dotsenko, Ismailov, and Umirbaev studied polynomial identities of Novikov algebras \cite{b16}; Shestakov
and Zhang proved analogs of It\^o's and Kegel's theorems for Novikov algebras \cite{b44}; Burde and
Dekimpe studied Novikov structures on Lie algebras \cite{b9}. The classification of small-dimensional
Novikov algebras was obtained in various papers by Abdurasulov, Bai, Bene\v{s}, Burde, de Graaf, Karimjanov,
Kaygorodov, Khudoyberdiyev, and Meng (see, for example, \cite{b7} and references therein).

The main example of a Novikov algebra is a commutative associative algebra with a derivation $D$
under the multiplication $x\cdot y=xD(y)$, and every Novikov algebra embeds into an algebra of this
kind \cite{b8}. A generalization of this construction to noncommutative associative algebras was given
by Kolesnikov and Sartayev \cite{b43}. In the present paper, we replace the derivation $D$ by a $\delta$-derivation
$\varphi$; the notion of $\delta$-derivations, which generalizes derivations, antiderivations, and centroid
elements, was introduced by Filippov \cite{b26}. We call the obtained algebras $\delta$-Novikov algebras;
their variety is defined by right commutativity and left symmetry of the $\delta$-associator. Algebras
defined by identities with the $\delta$-associator have a long history: for example, the anti-associator
plays the principal role in the definition of anti-associative \cite{b41}, Jacobi--Jordan-admissible \cite{b5},
and anti-pre-Lie \cite{b10,b35,b36} algebras, and $\delta$-associative and $\delta$-alternative algebras with a
trilinear map $\delta$ were considered in \cite{b11,b12,b38}.

The paper is organized as follows. In Section \ref{sec1}, we construct $\delta$-Novikov algebras $\Aphi$ from
commutative associative algebras $\A$ with a nonzero $\delta$-derivation $\varphi$ (Example \ref{ex5}) and show
that $\delta$-Novikov algebras share some properties of Novikov algebras: products and Lie products of
ideals are ideals (Lemma \ref{lem7}), unital $\delta$-Novikov algebras are commutative and associative
(Proposition \ref{prop8}), and the variety of $\delta$-Novikov algebras does not satisfy the Nielsen--Schreier
property (Corollary \ref{cor9}). Theorem \ref{thm11} gives the algebraic classification of $2$-dimensional
$\delta$-Novikov algebras. For every $\delta\neq1$, there are $2$-dimensional $\delta$-Novikov algebras that cannot
be embedded into a commutative associative algebra with a $\delta$-derivation (Proposition \ref{prop12}); this
answers negatively the question of such embeddings for $\delta\neq1$. The classification also shows a crucial
difference between Novikov and anti-Novikov algebras: by Zelmanov's theorem \cite{b47}, there are no
complex simple non-commutative finite-dimensional Novikov algebras, while there are such anti-Novikov
algebras. Following \cite{b35}, we introduce the Lie-admissible $\delta$-pre-Lie algebras and
classify the $2$-dimensional ones (Theorem \ref{thm19}), using the classification of \cite{b34} completed by the
family $\B D_4$ (the corrected version of \cite{b34}, which contains this family, is available on arXiv).

In Section \ref{sec2}, we give constructions of $\delta$-Novikov algebras from $\delta$-Novikov--Poisson algebras
(Lemma \ref{lem22}, Theorem \ref{thm23}, Corollary \ref{cor24}), and show that the Kantor product of the two
multiplications of a $\delta$-Novikov--Poisson algebra is a $\delta$-Novikov algebra (Theorem \ref{thm26}) and that
the tensor product of two $\delta$-Novikov--Poisson algebras is a $\delta$-Novikov--Poisson algebra (Theorem \ref{thm28}).

In Section \ref{sec3}, we show that for $\delta\neq1$ every identity of the commutator product of $\delta$-Novikov
algebras follows from anticommutativity, the Jacobi identity, and the metabelian identity (Theorem \ref{thm32});
together with \cite{b30}, this implies that every metabelian Lie algebra embeds into a $\delta$-Novikov algebra.
Note that for $\delta=1$ the commutator algebras of Novikov algebras need not be metabelian.

Section \ref{sec4} gives relations between $\delta$-derivations, (transposed) $\delta$-Poisson algebras, and
$\delta$-Novikov--Poisson algebras. We construct (transposed) $\delta$-Poisson algebras from commutative associative
algebras with a nonzero $\delta$-derivation (Theorem \ref{thm38}, Corollary \ref{cor42}) and prove that a $\delta$-Novikov--Poisson
algebra under the commutator product gives a transposed $(\delta+1)$-Poisson algebra (Theorem \ref{thm41}). For
every $\delta$, there are transposed $(\delta+1)$-Poisson algebras that cannot be embedded into $\delta$-Novikov--Poisson
algebras under the commutator product (Proposition \ref{prop43}); for $\delta=1$ this follows from the standard
identity of degree $5$, which holds in the commutator algebras of Novikov algebras. We also introduce
$\delta$-Gelfand--Dorfman algebras and prove that commutative $\delta$-Gelfand--Dorfman algebras give transposed
$(\delta+1)$-Poisson algebras (Theorem \ref{thm46}).

In Section \ref{sec5}, we prove that the Koszul dual of the operad of left $\delta$-Novikov
algebras is the operad of right $\delta$-Novikov algebras and that these operads are not Koszul (Theorem
\ref{thm48}; for $\delta=1$ this is a result of Dzhumadildaev \cite{b18}); the proof uses an explicit spanning set of the multilinear component of degree $4$ of the free
$\delta$-Novikov algebra (Lemma \ref{lem47}). 
The last section transfers known results on nilpotency and solvability of Novikov algebras to the
$\delta$-Novikov case. In a left $\delta$-Novikov algebra $N$ with $\delta\neq0$, the following statements are equivalent:
$N$ is right nilpotent, $N^2$ is nilpotent, and $N$ is solvable (Theorem \ref{thm52}); for $\delta=0$ this is no longer
true. Every minimal ideal $I$ of a $\delta$-Novikov algebra satisfies $I^2=0$ or is a simple algebra (Theorem \ref{thm49}).
For $\delta=1$ these results are due to Shestakov and Zhang \cite{b44}.

\medskip
\noindent\textbf{Notations.} We do not recall well-known definitions (Lie, Lie-admissible, Poisson, nilpotent,
and solvable algebras, etc.). For the commutator and the $\delta$-associator, we use the notations
\[
\begin{array}{c}
[x,y]:=x*y-y*x,\qquad (x,y,z)^*_\delta:=\delta(x*y)*z-x*(y*z),\\
\cyc_{x,y,z}f(x,y,z)=f(x,y,z)+f(y,z,x)+f(z,x,y).
\end{array}
\]
All algebras and vector spaces are considered over the field $\mathbb C$ of complex numbers. We almost
always assume that $\delta\neq1$, and algebras with two multiplications are assumed to be nontrivial, i.e.
both of their multiplications are nonzero.

\newpage

\section{\texorpdfstring{$\delta$}{δ}-Novikov and \texorpdfstring{$\delta$}{δ}-left-symmetric algebras}\label{sec1}
In this section, we introduce the notion of $\delta$-Novikov algebras, which are a generalization
of Novikov and bicommutative algebras with respect to the scalar parameter $\delta$. We
provide examples of this new structure with respect to $\delta$-derivations.

\begin{definition}
Let $\A$ be an algebra and $\delta$ be a fixed complex number. Then a linear map
$\varphi$ is a $\delta$-derivation if it satisfies
\begin{eqnarray*}
\varphi(xy)&=&\delta\big(\varphi(x)y+x\varphi(y)\big).
\end{eqnarray*}
\end{definition}

If $\delta=1$ (resp., $\delta=-1$) we have a derivation (resp., antiderivation\footnote{Recall that the notion of antiderivations plays an important role in the definition of mock-Lie algebras
\cite{b49} and (transposed) anti-Poisson algebras \cite{b1}; in the characterization of Lie algebras with identities \cite{b25}
and so on. In the anticommutative case, antiderivations coincide with reverse derivations defined by
Herstein in 1957 \cite{b29}.}). The notion of
$\delta$-derivations was introduced in a paper by Filippov \cite{b26} (see also references in \cite{b2,b32,b33,b48}).

\begin{definition}
Let $N$ be an algebra and $\delta$ be a fixed complex number. Then $N$ is called a
$\delta$-left-symmetric algebra\footnote{We say that $N$ is a $\delta$-associative algebra if $(x,y,z)_\delta=0$.} if the following identity holds:
\begin{eqnarray}\label{eq1}
\delta(xy)z-x(yz)&=&\delta(yx)z-y(xz),
\end{eqnarray}
A $\delta$-left-symmetric algebra $N$ is called a left $\delta$-Novikov algebra if the following identity
holds:
\begin{eqnarray}\label{eq2}
(xy)z&=&(xz)y.
\end{eqnarray}
\end{definition}

Let $\Var_\delta$ be a variety of algebras defined by the identities \eqref{eq1} and \eqref{eq2} for some $\delta$. Then
$\Var_0$ is the variety of bicommutative algebras\footnote{About recent results in bicommutative algebras, see \cite{b17}.}, and $\Var_1$ is the variety of Novikov algebras.
Hence, the varieties of $\delta$-Novikov algebras give a type of continuous deformation of
the variety of Novikov algebras to the variety of bicommutative algebras and vice versa.

\begin{proposition}
Let $\Vz$ be the subvariety of bicommutative algebras defined by the identity
\[(xy)z=(yx)z.\]
Then
\[\bigcap_\delta \Var_\delta=\Vz.\]
\end{proposition}

\begin{proof}
If $N\in\Var_{\delta_1}\cap\Var_{\delta_2}$ with $\delta_1\neq\delta_2$, then subtracting the two identities \eqref{eq1} gives
$(xy)z=(yx)z$, and then \eqref{eq1} gives $x(yz)=y(xz)$. Hence, already the intersection of two
varieties $\Var_{\delta_1}$ and $\Var_{\delta_2}$ with $\delta_1\neq\delta_2$ coincides with $\Vz$.
\end{proof}

Note that the variety $\Vz$ is not nilpotent, as can be seen in the following example.

\begin{example}\label{ex4}
Let $\A$ be a two-dimensional algebra with the multiplication table $e_1e_2=e_2$.
Then $\A$ is a $\delta$-Novikov algebra, moreover $\A\in\Vz$.
\end{example}

The classic example of the construction of Novikov algebras is based on a commutative
associative algebra with a nonzero derivation. We now present a generalization
of this example. Namely, we explain how some $\delta$-Novikov algebras can be constructed
using $\delta$-derivations.

\begin{example}\label{ex5}
Let $\A$ be a commutative associative algebra and let $\varphi$ be a $\delta$-derivation of
$\A$. Then $\Aphi$ is a $\delta$-Novikov algebra under multiplication
\begin{eqnarray}\label{eq3}
x\circ y&=&x\varphi(y).
\end{eqnarray}
\end{example}

One can define right $\delta$-Novikov algebras if the following identities hold true:
\begin{eqnarray*}
(xy)z-\delta x(yz)&=&(xz)y-\delta x(zy),\\
x(yz)&=&y(xz).
\end{eqnarray*}

Examples of right $\delta$-Novikov algebras can be given in the following way.

\begin{example}
Let $\A$ be a commutative associative algebra and let $\varphi$ be a $\delta$-derivation of
$\A$. Then $\A$ is a right $\delta$-Novikov algebra under multiplication
\[x\circ y=\varphi(x)y.\]
\end{example}

Similar to the case of ordinary right and left Novikov algebras, any right (resp., left)
$\delta$-Novikov algebra becomes a left (resp., right) $\delta$-Novikov algebra under the opposite multiplication
defined by $(x,y)\mapsto yx$. Unless stated otherwise, when we refer to $\delta$-Novikov
algebras, we mean left $\delta$-Novikov algebras.

$\delta$-Novikov algebras share many properties of Novikov algebras. Let us summarize
some common properties in the following observation.

\begin{lemma}\label{lem7}
Let $N$ be a $\delta$-Novikov algebra and assume that $I,J$ are two-sided ideals of $N$.
Then
\begin{enumerate}
\item[{\rm 1.}] $IJ$ is a two-sided ideal of $N;$
\item[{\rm 2.}] $[I,J]$ is a two-sided ideal of $N$ if $\delta\neq-1;$
\item[{\rm 3.}] the following identities hold in $N:$
\begin{eqnarray}
\label{eq4}[x,y]z+[y,z]x+[z,x]y&=&0,\\
\label{eq5}x[y,z]+y[z,x]+z[x,y]&=&0,
\end{eqnarray}
i.e. $N$ is Lie-admissible.
\end{enumerate}
\end{lemma}

\begin{proof}
Let $x\in I$, $y\in J$ and $a\in N$, then
\begin{alignat*}{2}
(xy)a&=(xa)y &\qquad&\in\ \ IJ,\\
a(xy)&=\delta(ax)y-\delta(xa)y+x(ay) &&\in\ \ IJ,
\end{alignat*}
hence $IJ$ is an ideal.
\begin{alignat*}{2}
[x,y]a&=\tfrac{1}{\delta+1}\big([xa,y]+[x,ya]\big) &\qquad&\in\ \ [I,J],\\
a[x,y]&=[x,y]a+[a,[x,y]] &&\in\ \ [I,J],
\end{alignat*}
hence $[I,J]$ is an ideal if $\delta\neq-1$. The condition $\delta\neq-1$ is essential: in the anti-Novikov
algebra $\B N_{11}$ of Theorem \ref{thm11} one has $[N,N]=\langle e_2-e_1\rangle$, while $e_1(e_2-e_1)=-2e_1\notin[N,N]$. Here the identity $(\delta+1)[x,y]a=[xa,y]+[x,ya]$ follows from \eqref{eq1} and
\eqref{eq2}, since the difference of its two sides equals \begin{center}
$(x,y,a)_\delta-(y,x,a)_\delta+\big((xy)a-(xa)y\big)-\big((yx)a-(ya)x\big)$;
\end{center}
and in the second line $[a,[x,y]]=[[a,x],y]+[x,[a,y]]\in[I,J]$ by the Jacobi identity,
which holds by part 3.

The first identity from part 3 follows from the right commutative identity \eqref{eq2} and the
second identity follows from \eqref{eq1} and \eqref{eq4}:
\begin{eqnarray*}
x[y,z]+y[z,x]+z[x,y]&=&x(yz)-x(zy)+y(zx)-y(xz)+z(xy)-z(yx)\\
&=&\delta\big((xy)z-(yx)z+(zx)y-(xz)y+(yz)x-(zy)x\big)\ \ =\ \ 0.
\end{eqnarray*}

Finally, in every algebra $[[x,y],z]+[[y,z],x]+[[z,x],y]$ equals the left side of \eqref{eq4} minus
the left side of \eqref{eq5}, so $N$ is Lie-admissible.
\end{proof}

\begin{proposition}\label{prop8}
If a $\delta$-Novikov algebra $N$ has a unity element, then $N$ is commutative
and associative.
\end{proposition}

\begin{proof}
By taking $x=1$ in \eqref{eq2}, we have commutativity, and then \eqref{eq2} gives associativity:
$x(yz)=(yz)x=(yx)z=(xy)z$. Note that only a left unity is needed here.
\end{proof}

In \cite{b15}, it was proved that the only nontrivial variety of non-associative algebras that
is both a Nielsen--Schreier variety (i.e. every subalgebra of a free algebra is free) and $I^2$
is an ideal whenever $I$ is an ideal is the variety of Lie algebras. By Lemma \ref{lem7} (part 1 with
$J=I$: $I^2$ is an ideal for every ideal $I$) and the main result in \cite{b15} we have

\begin{corollary}\label{cor9}
$\Var_\delta$ does not satisfy the Nielsen--Schreier property.
\end{corollary}

\begin{theorem}\label{thm10}
Let $\A$ be a commutative associative algebra and let $\varphi$ be a $\delta$-derivation of
$\A$. Define the operation $\circ$ by \eqref{eq3}. Then
\begin{enumerate}
\item[{\rm 1.}] if $\delta\notin\big\{0,\frac12,1\big\}$, then $\Aphi$ satisfies $\Aphi^4=0$, i.e. all products of four elements of $\Aphi$ are
zero.

\item[{\rm 2.}] if $\delta=0$, then $\Aphi$ is $2$-step left nilpotent, but not necessarily right nilpotent.

\item[{\rm 3.}] if $\delta=\frac12$, then $\Aphi$ is not necessarily nilpotent, solvable or left nilpotent.

\item[{\rm 4.}] if $\delta=\frac12$, then $\Aphi\in\Vz$, i.e. $\Aphi$ satisfies $(x\circ y)\circ z=(y\circ x)\circ z$ and $x\circ(y\circ z)=y\circ(x\circ z)$;
in particular, $\Aphi$ is a $\gamma$-Novikov algebra for every $\gamma$.
\end{enumerate}
\end{theorem}

\begin{proof}
First, let $\delta\notin\big\{0,\frac12,1\big\}$. It is easy to see that
\begin{eqnarray*}
\delta\varphi(x)yz+\delta^2x\varphi(y)z+\delta^2xy\varphi(z)&=&\varphi(xyz)\ \ =\ \ \delta^2\varphi(x)yz+\delta^2x\varphi(y)z+\delta xy\varphi(z),
\end{eqnarray*}
then $(\delta^2-\delta)\varphi(x)yz\ =\ (\delta^2-\delta)xy\varphi(z)$, i.e. $\varphi(x)yz\ =\ xy\varphi(z)$.\\
On the other hand,
\begin{eqnarray*}
\varphi((xy)(zt))&=&4\delta^2\varphi(x)yzt,\\
\varphi(x(y(zt)))&=&(\delta+\delta^2+2\delta^3)\varphi(x)yzt.
\end{eqnarray*}

Since $4\delta^2-(\delta+\delta^2+2\delta^3)=-\delta(2\delta-1)(\delta-1)\neq0$, the latter implies $\varphi(x)yzt\ =\ 0$ and
$\varphi(xyzt)\ =\ 0$. It follows that any product of $4$ elements in $\Aphi$ is zero.

Second, let $\delta=0$. Hence, $\varphi(xy)=0$ and $x\circ(y\circ z)=0$. On the other hand, we consider
a $2$-dimensional algebra with the multiplication $e_1e_1=e_1$ and $\varphi$ defined as $\varphi(e_2)=e_1$. $\varphi$
is a $0$-derivation, but $(\ldots(e_1\circ e_2)\ldots)\circ e_2=e_1\neq0$. In fact, here $e_1\circ e_2=e_1$ and all other
products of basis elements are zero, i.e. $\Aphi$ is the algebra $\B N_{03}$ of Theorem \ref{thm11}.

Third, the basic example of $\frac12$-derivations is the identity map $\mathrm{Id}$. Hence, taking $\varphi=\mathrm{Id}$
for a non-nilpotent algebra $\A$ (for instance, a nonzero unital algebra), we have that $\Aphi$
is isomorphic to $\A$ and it is non-nilpotent, non-solvable and non-left nilpotent.

Fourth, let $\delta=\frac12$. Since $\delta^2-\delta\neq0$, the first computation gives $\varphi(x)yz=xy\varphi(z)$. Applying
it with $z$ replaced by $\varphi(z)$, we obtain 
\begin{center}
$(x\circ y)\circ z=x\varphi(y)\varphi(z)=xy\varphi^2(z)=y\varphi(x)\varphi(z)=
(y\circ x)\circ z$
\end{center} and 
\begin{center}$x\circ(y\circ z)=x\varphi\big(y\varphi(z)\big)=\frac12\big(x\varphi(y)\varphi(z)+xy\varphi^2(z)\big)=xy\varphi^2(z)$,
\end{center} which is symmetric
in $x$ and $y$. Together with the right commutativity of $\Aphi$, this gives $\Aphi\in\Vz$.
\end{proof}

The classification of complex two-dimensional algebras was
established in \cite{b34}. We present a classification of two-dimensional $\delta$-Novikov algebras. It
can be obtained from the classification of all two-dimensional algebras in \cite{b34}, completed
by the family $\B D_4$ (see the second remark after Theorem \ref{thm19}): for every family one solves
the polynomial equations on the structure constants given by the identities \eqref{eq1} and \eqref{eq2};
the second column of the table gives the corresponding algebra of \cite{b34} (or of the family
$\B D_4$). The details of this computation are given after Theorem \ref{thm19}.

\begin{theorem}\label{thm11}
Let $N$ be a nontrivial $2$-dimensional $\delta$-Novikov algebra, then it is isomorphic
to one algebra listed below \big(the number $\gamma$ in the third column means that it is a
$\gamma$-Novikov algebra\big).
\end{theorem}

\begin{longtable}{|l|l|c|l|l|l|l|}
\hline
$\B N_{01}$ & $\B A_1(0)$ & $1$ & $e_1e_1=e_1+e_2$ & $e_1e_2=0$ & $e_2e_1=e_2$ & $e_2e_2=0$\\ \hline
$\B N_{02}$ & $\B A_3$ & $\delta$ & $e_1e_1=e_2$ & $e_1e_2=0$ & $e_2e_1=0$ & $e_2e_2=0$\\ \hline
$\B N_{03}$ & $\B B_2(0)$ & $0$ & $e_1e_1=0$ & $e_1e_2=e_1$ & $e_2e_1=0$ & $e_2e_2=0$\\ \hline
$\B N_{04}$ & $\B B_2(1)$ & $\delta$ & $e_1e_1=0$ & $e_1e_2=0$ & $e_2e_1=e_1$ & $e_2e_2=0$\\ \hline
$\B N_{05}$ & $\B D_1(0,0)$ & $0$ & $e_1e_1=e_1$ & $e_1e_2=e_1$ & $e_2e_1=0$ & $e_2e_2=0$\\ \hline
$\B N_{06}$ & $\B D_2(0,0)$ & $\delta$ & $e_1e_1=e_1$ & $e_1e_2=0$ & $e_2e_1=0$ & $e_2e_2=0$\\ \hline
$\B N_{07}^{\delta}$ & $\B D_2(0,\delta^{-1})$, $\delta\neq0$ & $\delta$ & $e_1e_1=e_1$ & $e_1e_2=0$ & $e_2e_1=\delta^{-1}e_2$ & $e_2e_2=0$\\ \hline
$\B N_{08}$ & $\B D_2(1,1)$ & $\delta$ & $e_1e_1=e_1$ & $e_1e_2=e_2$ & $e_2e_1=e_2$ & $e_2e_2=0$\\ \hline
$\B N_{09}^{\alpha}$ & $\B D_2(\alpha,1)$, $\alpha\notin\{0,1\}$ & $1$ & $e_1e_1=e_1$ & $e_1e_2=\alpha e_2$ & $e_2e_1=e_2$ & $e_2e_2=0$\\ \hline
$\B N_{10}$ & $\B E_1(0,0,0,0)$ & $\delta$ & $e_1e_1=e_1$ & $e_1e_2=0$ & $e_2e_1=0$ & $e_2e_2=e_2$\\ \hline
$\B N_{11}$ & $\B E_1(-1,0,0,-1)$ & $-1$ & $e_1e_1=e_1$ & $e_1e_2=-e_1$ & $e_2e_1=-e_2$ & $e_2e_2=e_2$\\ \hline
\end{longtable}

\begin{remark}
The algebra $\B N_{07}^{\delta}$ is defined only for $\delta\neq0$, and the parameters $\alpha=0$ and $\alpha=1$
are excluded in $\B N_{09}^{\alpha}$, since $\B N_{09}^{0}$ coincides with $\B N_{07}^{1}$ and $\B N_{09}^{1}$ coincides with $\B N_{08}$. The algebra
$e_1e_1=e_1$, $e_1e_2=0$, $e_2e_1=e_1$ does not appear in the table, since it is isomorphic to $\B N_{05}$:
in $\B N_{05}$, the basis $f_1=e_1$, $f_2=e_1-e_2$ satisfies $f_1f_1=f_1$, $f_1f_2=0$, $f_2f_1=f_1$, $f_2f_2=0$. A
direct verification shows that each algebra of the table satisfies the identities of $\gamma$-Novikov
algebras for the value $\gamma$ given in the third column. For $\delta=1$ the pair $(0,\delta^{-1})=(0,1)$ does not
satisfy the restriction of the family $\B D_2$ in \cite{b34}: the algebra $\B N_{07}^{1}$ has two nonzero
idempotents $e_1$ and $e_1+e_2$, and in the basis $e_1,\ e_1+e_2$ it is the algebra $\B E_5(0)$ of \cite{b34}.
\end{remark}

The above-mentioned theorem gives a crucial difference between Novikov and anti-Novikov
algebras. By  Zelmanov's Theorem \cite{b47}, there are no complex simple finite-dimensional
non-commutative Novikov algebras, but in the anti-Novikov case, we have
a different situation: the algebra $\B N_{11}$ of Theorem \ref{thm11} is simple and non-commutative. Indeed, $e_1e_2=-e_1\neq-e_2=e_2e_1$; if $0\neq v=v_1e_1+v_2e_2$ lies in an ideal
$I$, then so do $e_1v=(v_1-v_2)e_1$, $e_2v=(v_2-v_1)e_2$ and $ve_1=v_1e_1-v_2e_2$, and in both cases
$v_1\neq v_2$ and $v_1=v_2$ one obtains $e_1,e_2\in I$.

It is known that each Novikov algebra can be embedded in a suitable commutative
associative algebra with a derivation under the multiplication given by \eqref{eq3} \cite{b8}. 
The analogous question arises for $\delta$-Novikov algebras.

The following proposition shows that for $\delta\neq1$ the answer is negative.

\begin{proposition}\label{prop12}
Let $\delta\neq1$. Then there is a $2$-dimensional $\delta$-Novikov algebra that cannot
be embedded into an algebra $\Aphi$ from Example \ref{ex5}: one can take $\B N_{06}$ if $\delta\neq\frac12$ and $\B N_{07}^{1/2}$ if
$\delta=\frac12$.
\end{proposition}

\begin{proof}
If $\delta\notin\{0,\frac12,1\}$, then all products of four elements vanish in $\Aphi$ by Theorem \ref{thm10} (1),
while $((e_1e_1)e_1)e_1=e_1\neq0$ in $\B N_{06}$. If $\delta=0$, then $x\circ(y\circ z)=0$ in $\Aphi$, while $e_1(e_1e_1)=e_1\neq0$
in $\B N_{06}$. If $\delta=\frac12$, then $\Aphi$ satisfies $(xy)z=(yx)z$ by Theorem \ref{thm10} (4), and so does every
subalgebra of $\Aphi$, while in $\B N_{07}^{1/2}$ we have $(e_2e_1)e_1=4e_2\neq0=(e_1e_2)e_1$.
\end{proof}

On the other hand, for $\delta=0$ the algebra $\B N_{03}$, which is not right nilpotent, is itself of
the form $\Aphi$ (see the proof of Theorem \ref{thm10} (2)); so non-nilpotent $\delta$-Novikov algebras can
be embedded into algebras $\Aphi$. The following question remains open.

\noindent\textbf{Open question}. Let $\delta\neq1$. Which $\delta$-Novikov algebras can be embedded into a commutative
associative algebra with a $\delta$-derivation? Is every $\delta$-Novikov algebra satisfying the
identities of $\Aphi$ given in Theorem \ref{thm10} ($\A^4=0$ for $\delta\notin\{0,\frac12,1\}$, $x(yz)=0$ for $\delta=0$, $\A\in\Vz$ for
$\delta=\frac12$) embeddable?

\begin{definition}
Let $P$ be an algebra and $\delta$ be a fixed complex number. Then $P$ is called a
$\delta$-pre-Lie algebra\footnote{For $\delta=-1$ we have the definition of anti-pre-Lie algebras introduced in \cite{b35}.} if it satisfies \eqref{eq1} and \eqref{eq4}.
\end{definition}

\begin{proposition}
Let $P$ be a $\delta$-Novikov algebra, then $P$ is $\delta$-pre-Lie. Let $P$ be a $\delta$-pre-Lie
algebra, then $P$ is Lie-admissible. Let $P$ be a $\delta$-left-symmetric algebra $(\delta\neq1)$, then $P$ is
Lie-admissible if and only if it satisfies \eqref{eq4}; if moreover $\delta\neq0$, this holds if and only if $P$
satisfies \eqref{eq5}.
\end{proposition}

\begin{proof}
Follows from Lemma \ref{lem7}. More precisely, in every algebra the Jacobian of the commutator
equals the left side of \eqref{eq4} minus the left side of \eqref{eq5}, and in a $\delta$-left-symmetric
algebra the left side of \eqref{eq5} equals $\delta$ times the left side of \eqref{eq4} (see the proof of Lemma \ref{lem7});
so the Jacobian equals $(1-\delta)$ times the left side of \eqref{eq4}.
\end{proof}

\begin{remark}
For $\delta=0$ the identity \eqref{eq5} does not imply Lie-admissibility: the $4$-dimensional
algebra $e_1e_2=e_4$, $e_4e_3=e_4$ is $0$-left-symmetric \textup(all products $x(yz)$ vanish\textup) and satisfies
\eqref{eq5}, but $[[e_1,e_2],e_3]+[[e_2,e_3],e_1]+[[e_3,e_1],e_2]=e_4\neq0$.
\end{remark}

\begin{definition}\label{def15}
Let $\A$ be an algebra and $\delta$ be a fixed complex number. Then a linear map
$R$ is a $\delta$-Rota--Baxter operator\footnote{For $\delta=1$ we recover the definition of Rota--Baxter operator and for $\delta=-1$
we recover the definition of anti-Rota--Baxter operator \cite{b10,b35}. Note that the identity of a Rota--Baxter
operator of weight $\lambda$ is often written with $+\lambda xy$; in that convention our weight $1$ is the weight $-1$.} of weight $\lambda\in\{0,1\}$ if it satisfies
\begin{eqnarray*}
R(x)R(y)&=&\delta R\big(R(x)y+xR(y)-\lambda xy\big).
\end{eqnarray*}
\end{definition}

\begin{proposition}
Let $\mathfrak g$ be a Lie algebra, $R$ be a $\delta$-Rota--Baxter operator of weight $0$ and
let a new multiplication $\star$ be defined by $x\star y=R(x)y$. Then, $(\mathfrak g,\star)$ is a $\delta$-left-symmetric
algebra; if $\delta\neq0$, then $(\mathfrak g,\star)$ is a $\delta$-pre-Lie algebra if and only if $\cyc_{x,y,z}(R(x)R(y))z=0$.
\end{proposition}

\begin{proof}
It is easy to see that
\begin{eqnarray*}
(x,y,z)^\star_\delta&=&\delta R(R(x)y)z-R(x)(R(y)z)\\
&=&(R(x)R(y))z-\delta R(xR(y))z-(R(x)R(y))z-R(y)(R(x)z)\ \ =\ \ (y,x,z)^\star_\delta.
\end{eqnarray*}

On the other hand,
\begin{eqnarray*}
\cyc_{x,y,z}[x,y]\star z&=&R(R(x)y-R(y)x)z+R(R(y)z-R(z)y)x+R(R(z)x-R(x)z)y\\
&=&\delta^{-1}\big((R(x)R(y))z+(R(y)R(z))x+(R(z)R(x))y\big).
\end{eqnarray*}

Here the assumption $\delta\neq0$ is essential: for $\delta=0$ one has $R(x)R(y)=0$, so the condition
$\cyc_{x,y,z}(R(x)R(y))z=0$ always holds, while $(\mathfrak g,\star)$ need not be $0$-pre-Lie.
\end{proof}

\begin{corollary}
Let $\A$ be an associative algebra, $R$ be a $\delta$-Rota--Baxter operator of weight
$0$ and let a new multiplication $\star$ be defined by $x\star y=R(x)y-yR(x)$. Then, $(\A,\star)$ is a
$\delta$-left-symmetric algebra.
\end{corollary}

\begin{proposition}
Let $\A$ be a $\delta$-associative algebra, $R$ be a Rota--Baxter operator of weight
$1$ \textup(i.e. a $1$-Rota--Baxter operator of weight $1$ in the sense of Definition \ref{def15}: $R(x)R(y)=
R\big(R(x)y+xR(y)-xy\big)$; the proof uses $R(x\star y)=R(x)R(y)$\textup) and let a new multiplication $\star$
be defined by $x\star y=R(x)y+xR(y)-xy$. Then, $(\A,\star)$ is also a $\delta$-associative algebra.
\end{proposition}

\begin{proof}
It is easy to see that
\begin{longtable}{l}
$(x,y,z)^\star_\delta\ \ =\ \ \delta(x\star y)\star z-x\star(y\star z)\ \ =$\\
$\delta\big(R(R(x)y+xR(y)-xy)z+(R(x)y+xR(y)-xy)R(z)-(R(x)y)z-(xR(y))z+(xy)z\big)$\\
$ -\big(R(x)(R(y)z+yR(z)-yz)+xR(R(y)z+yR(z)-yz)-x(R(y)z)-x(yR(z))+x(yz)\big)\ \ =$\\
$ \quad(R(x),R(y),z)_\delta+(R(x),y,R(z))_\delta+(x,R(y),R(z))_\delta-$\\
\multicolumn{1}{r}{$(x,y,R(z))_\delta-(R(x),y,z)_\delta-(x,R(y),z)_\delta+(x,y,z)_\delta\ \ =\ \ 0.$}
\end{longtable}

\end{proof}

\begin{theorem}\label{thm19}
Each $2$-dimensional $\delta$-left-symmetric algebra is $\delta$-pre-Lie. Let $P$ be a
nontrivial $2$-dimensional $\delta$-pre-Lie algebra, then it is isomorphic to a
$\delta$-Novikov algebra, listed in Theorem \ref{thm11} or one algebra listed below \textup(the number
$\gamma$ in the third column means that it is a $\gamma$-pre-Lie algebra\textup).
\end{theorem}

\begin{longtable}{|l|l|c|l|l|}
\hline
$\B P_{01}$ & $\B A_1(1)$ & $\delta$ & $e_1e_1=e_1+e_2$ & $e_1e_2=e_2$\\*
 & & & $e_2e_1=0$ & $e_2e_2=0$\\ \hline
$\B P_{02}^{\delta}$ & $\B A_1(\frac{\delta-1}{2\delta-1})$, & $\delta$ & $e_1e_1=e_1+e_2$ & $e_1e_2=\frac{\delta-1}{2\delta-1}e_2$\\*
 & $\delta\notin\{0,\frac12,1\}$ & & $e_2e_1=\frac{\delta}{2\delta-1}e_2$ & $e_2e_2=0$\\ \hline
$\B P_{03}$ & $\B A_2$ & $\frac12$ & $e_1e_1=e_2$ & $e_1e_2=e_2$\\*
 & & & $e_2e_1=-e_2$ & $e_2e_2=0$\\ \hline
$\B P_{04}^{\delta}$ & $\B B_2(\frac{\delta}{2\delta-1})$, & $\delta$ & $e_1e_1=0$ & $e_1e_2=\frac{\delta-1}{2\delta-1}e_1$\\*
 & $\delta\notin\{0,\frac12,1\}$ & & $e_2e_1=\frac{\delta}{2\delta-1}e_1$ & $e_2e_2=0$\\ \hline
$\B P_{05}$ & $\B B_3$ & $\frac12$ & $e_1e_1=0$ & $e_1e_2=e_2$\\*
 & & & $e_2e_1=-e_2$ & $e_2e_2=0$\\ \hline
$\B P_{06}$ & $\B C(1,0)$ & $0$ & $e_1e_1=e_2$ & $e_1e_2=0$\\*
 & & & $e_2e_1=e_1$ & $e_2e_2=e_2$\\ \hline
$\B P_{07}$ & $\B D_1(1,1)$ & $0$ & $e_1e_1=e_1$ & $e_1e_2=e_2$\\*
 & & & $e_2e_1=e_1-e_2$ & $e_2e_2=0$\\ \hline
$\B P_{08}^{\alpha}$ &  $\B D_2(\alpha,0)$, & $\delta$ & $e_1e_1=e_1$ & $e_1e_2=\alpha e_2$\\*
 & $\alpha\notin\{0,1\}$ & & $e_2e_1=0$ & $e_2e_2=0$\\ \hline
$\B P_{09}^{\beta,\delta}$ & $\B D_2(\frac{\delta\beta-1}{\delta-1},\beta)$, & $\delta$ & $e_1e_1=e_1$ & $e_1e_2=\frac{\delta\beta-1}{\delta-1}e_2$\\*
  & $\beta\notin\{0,1\}$, $\delta\neq1$, & & $e_2e_1=\beta e_2$ & $e_2e_2=0$\\*
 & $\delta\beta\neq1$ & & & \\ \hline
$\B P_{10}^{\beta}$ & $\B E_1(0,\beta,\frac1\beta,0)$ & $0$ & $e_1e_1=e_1$ & $e_1e_2=\beta e_2$\\*
 & $\beta\notin\{0,1\}$, $\B P_{10}^{\beta}\cong\B P_{10}^{1/\beta}$ & & $e_2e_1=\frac1\beta e_1$ & $e_2e_2=e_2$\\ \hline
$\B P_{11}^{\alpha,\beta}$ & $\B E_1(\alpha,\beta,1,0)$, & $0$ & $e_1e_1=e_1$ & $e_1e_2=\alpha e_1+\beta e_2$\\*
 & $(\alpha,\beta)\notin\{(0,1),(1,0)\}$, & & $e_2e_1=e_1$ & $e_2e_2=e_2$\\*
 & $\B P_{11}^{\alpha,\beta}\cong\B P_{11}^{\alpha,\frac{\alpha+\beta-1}{\alpha\beta+\beta-1}}$ & & & \\ \hline
$\B P_{12}$ & $\B E_1(0,1,1,0)$ & $\delta$ & $e_1e_1=e_1$ & $e_1e_2=e_2$\\*
 & & & $e_2e_1=e_1$ & $e_2e_2=e_2$\\ \hline
$\B P_{13}^{\delta}$ & $\B E_1(\delta,1+\delta,1+\delta,\delta)$ & $\delta$ & $e_1e_1=e_1$ & $e_1e_2=\delta e_1+(1+\delta)e_2$\\*
 & $\delta\notin\{-1,0\}$ & & $e_2e_1=(1+\delta)e_1+\delta e_2$ & $e_2e_2=e_2$\\ \hline
$\B P_{14}$ & $\B D_4(1,1)$ & $0$ & $e_1e_1=e_1$ & $e_1e_2=e_2$\\*
 & & & $e_2e_1=e_1$ & $e_2e_2=0$\\ \hline
$\B P_{15}$ & $\B D_4(0,1)$ & $0$ & $e_1e_1=e_1$ & $e_1e_2=e_1+e_2$\\*
 & & & $e_2e_1=0$ & $e_2e_2=0$\\ \hline
\end{longtable}

\begin{remark}
Every $2$-dimensional algebra satisfies \eqref{eq4} identically \textup(a direct computation\textup),
which proves the first claim. A direct verification shows that each algebra of the table is
$\gamma$-left-symmetric for the indicated $\gamma$.
The restrictions on the parameters remove the coincidences
$\B P_{02}^{0}=\B P_{01}$, $\B P_{08}^{1}\cong\B P_{12}$ \textup{(}take the basis
$e_1,\ e_1+e_2$\textup{)}, $\B P_{09}^{0,\delta}=\B P_{08}^{1/(1-\delta)}$, $\B P_{09}^{1,\delta}=\B N_{08}$ and, for $\delta\neq0$,
$\B P_{09}^{\delta^{-1},\delta}=\B N_{07}^{\delta}$; for $\delta=0$ the condition $\delta\beta\neq1$ is void, and
$\B P_{09}^{\beta,0}=\B D_2(1,\beta)$, $\beta\notin\{0,1\}$. If $\delta\notin\{0,\frac12,1\}$ and $\beta=\frac{\delta}{2\delta-1}$, then the
pair $\big(\frac{\delta\beta-1}{\delta-1},\beta\big)$ does not satisfy the restriction of the family $\B D_2$ in \cite{b34}:
in this case $\B P_{09}^{\beta,\delta}$ has two nonzero idempotents $e_1$ and $e_1+e_2$, and in the basis
$e_1,\ e_1+e_2$ it is the algebra $\B E_5\big(\frac{\delta-1}{2\delta-1}\big)$ of \cite{b34}.
Inside the families $\B P_{10}$ and $\B P_{11}$ there remain the identifications
$\B P_{10}^{\beta}\cong\B P_{10}^{1/\beta}$ \textup{(}the permutation
$e_1\leftrightarrow e_2$\textup{)} and, for $\alpha\neq0$ and $\alpha\beta+\beta\neq1$,
\[
\B P_{11}^{\alpha,\beta}\cong\B P_{11}^{\alpha,\beta'},\qquad
\beta'=\frac{\alpha+\beta-1}{\alpha\beta+\beta-1}
\]
\textup{(}the basis $f_1=\frac{\alpha e_1+(\beta-1)e_2}{\alpha\beta+\beta-1},\ f_2=e_2$ of
$\B P_{11}^{\alpha,\beta}$ has the multiplication table of
$\B P_{11}^{\alpha,\beta'}$\textup{)}; the map $\beta\mapsto\beta'$ is an involution.
\end{remark}

\begin{remark}
The algebras $\B P_{14}$ and $\B P_{15}$ do not appear in the published version of \cite{b34}. In the proof of \cite[Lemma
7]{b34}, the case $c_{12}^1+c_{21}^1\neq0$, $c_{12}^2+c_{21}^2=1$ is lost: the second idempotent
constructed there coincides with $e_1$ when $c_{12}^2+c_{21}^2=1$. This case gives the family
\[
\B D_4(\alpha,\beta):\quad e_1e_1=e_1,\quad e_1e_2=(1-\alpha)e_1+\beta e_2,\quad
e_2e_1=\alpha e_1+(1-\beta)e_2,\quad e_2e_2=0.
\]
The algebras $\B D_4(\alpha,\beta)$ are pairwise non-isomorphic, and none of them is
isomorphic to an algebra of \cite[Table 1]{b34} \textup(in $\B D_4(\alpha,\beta)$, the element $e_1$ is the
only nonzero idempotent, the elements with zero square form the line $\mathbb C e_2$, and
$e_1e_2+e_2e_1=e_1+e_2\notin\mathbb C e_2$; no algebra of \cite[Table 1]{b34} has these three
properties\textup). We have $\B P_{14}=\B D_4(1,1)$ and $\B P_{15}=\B D_4(0,1)$. The algebra $\B D_4(\alpha,\beta)$ is $\gamma$-left-symmetric
if and only if $\gamma=0$, $\beta=1$ and $\alpha\in\{0,1\}$; in particular, no algebra
$\B D_4(\alpha,\beta)$ is $\gamma$-Novikov. Hence the family $\B D_4$ does not change
Theorem \ref{thm11}, and it adds exactly the algebras $\B P_{14}$ and $\B P_{15}$ to Theorem \ref{thm19}.
The corrected version of \cite{b34}, which contains the family $\B D_4$, is available on arXiv
\textup(\arxiv{1701.08233}\textup).
\end{remark}

\noindent\textit{Proof of the completeness of Theorems \ref{thm11} and \ref{thm19}.}
By the corrected version of \cite{b34} (see the remark above), every nontrivial $2$-dimensional algebra over $\mathbb C$ either is isomorphic to
exactly one algebra of the families $\B A_1$, $\B A_2$, $\B A_3$, $\B A_4$, $\B B_1$, $\B B_2$, $\B B_3$, $\B C$, $\B D_1$, $\B D_2$, $\B D_3$, $\B D_4$
\textup(with the parameters restricted as in \cite[Table 1]{b34}\textup), or belongs to the class $\mathbf E$, i.e. has two
different nonzero idempotents. Two different nonzero idempotents are linearly independent, so every
algebra of the class $\mathbf E$ has a basis $e_1,e_2$ in which its multiplication table is
\[
\B E_1(a,b,c,d):\quad e_1e_1=e_1,\quad e_1e_2=ae_1+be_2,\quad e_2e_1=ce_1+de_2,\quad e_2e_2=e_2
\]
for some $(a,b,c,d)\in\mathbb C^4$ \textup(here, unlike \cite{b34}, we do not restrict the parameters; note that
$\B E_5(\alpha)=\B E_1(1-\alpha,\alpha,\alpha,1-\alpha)$\textup). Since every $2$-dimensional algebra satisfies \eqref{eq4},
a $2$-dimensional algebra is $\delta$-pre-Lie if and only if it satisfies \eqref{eq1}. Writing \eqref{eq1} for the basis
elements, one obtains polynomial equations on the structure constants; their solutions are listed in
the following table, where the last column gives the algebra of Theorem \ref{thm11} or Theorem \ref{thm19}
isomorphic to the corresponding solution.

{\small
\begin{longtable}{|p{2cm}|p{6cm}|p{6cm}|}
\hline
family & solutions of \eqref{eq1} & algebra\\ \hline\endhead
$\B A_1(\alpha)$ & $\alpha=1$ & $\B P_{01}$\\*
 & $(2\alpha-1)\delta=\alpha-1$, i.e. $\alpha=\frac{\delta-1}{2\delta-1}$, $\delta\neq\frac12$ & $\B P_{02}^{\delta}$ $(\delta\neq0,1)$; $\B P_{01}$ $(\delta=0)$; $\B N_{01}$ $(\delta=1)$\\ \hline
$\B A_2$ & $\delta=\frac12$ & $\B P_{03}$\\ \hline
$\B A_3$ & every $\delta$ & $\B N_{02}$\\ \hline
$\B A_4(\alpha)$  & no solutions & \\ \hline

$\B B_1(\alpha)$  & no solutions & \\ \hline

$\B B_2(\alpha)$ & $\alpha=1$ & $\B N_{04}$\\*
 & $(2\alpha-1)\delta=\alpha$, i.e. $\alpha=\frac{\delta}{2\delta-1}$, $\delta\neq\frac12$ & $\B P_{04}^{\delta}$ $(\delta\neq0,1)$; $\B N_{03}$ $(\delta=0)$; $\B N_{04}$ $(\delta=1)$\\ \hline
$\B B_3$ & $\delta=\frac12$ & $\B P_{05}$\\ \hline
$\B C(\alpha,\beta)$ & $\delta=0$, $(\alpha,\beta)=(1,0)$ & $\B P_{06}$\\ \hline
$\B D_1(\alpha,\beta)$ & $\delta=0$, $(\alpha,\beta)\in\{(0,0),(1,0)\}$ & $\B N_{05}$ \textup($\B D_1(1,0)\cong\B D_1(0,0)$\textup)\\*
 & $\delta=0$, $(\alpha,\beta)=(1,1)$ & $\B P_{07}$\\ \hline
$\B D_2(\alpha,\beta)$ & $\beta=0$ & $\B N_{06}$ $(\alpha=0)$; $\B P_{08}^{\alpha}$ $(\alpha\neq0)$\\ 
 $\alpha+\beta\neq1$  & $\beta\neq0$, $(\delta-1)\alpha=\delta\beta-1$ & $\delta\neq1$: $\B N_{08}$ $(\beta=1)$; $\B N_{07}^{\delta}$ $(\beta=\delta^{-1})$;\qquad $\B P_{09}^{\beta,\delta}$ \textup(otherwise\textup);\newline $\delta=1$: $\B N_{08}$ $(\alpha=1)$; $\B N_{09}^{\alpha}$ $(\alpha\neq1)$\\ \hline

 $\B D_3(\alpha,\beta)$ & no solutions & \\ \hline

$\B D_4(\alpha,\beta)$ & $\delta=0$, $\beta=1$, $\alpha\in\{0,1\}$ & $\B P_{15}$ $(\alpha=0)$; $\B P_{14}$ $(\alpha=1)$\\ \hline
$\B E_1(a,b,c,d)$ & $(0,0,0,0)$, $(0,1,0,1)$, $(1,0,1,0)$ & $\B N_{10}$\\*
 & $(0,1,1,0)$ & $\B P_{12}$\\*
 & $\B E_5(\alpha)$ with $(2\alpha-1)\delta=\alpha-1$, $\delta\neq\frac12$ & $\B P_{09}^{\delta/(2\delta-1),\delta}$ $(\delta\neq0,1)$; $\B P_{12}$ $(\delta=0)$;\qquad  $\B N_{07}^{1}$ $(\delta=1)$\\*
 & $(\delta,1+\delta,1+\delta,\delta)$; for $\delta\neq-1$ also \qquad$\big(0,1+\delta,\frac{1}{1+\delta},\delta\big)$ and $\big(\delta,\frac{1}{1+\delta},1+\delta,0\big)$ & $\B P_{13}^{\delta}$ $(\delta\neq-1,0)$; $\B N_{11}$ $(\delta=-1)$;\qquad \qquad$\B P_{12}$ $(\delta=0)$\\*
 & $\delta=0$: $(a,b,1,0)$ & $\B P_{11}^{a,b}$; $\B P_{12}$ if $(a,b)=(0,1)$;\qquad \qquad$\B N_{10}$ if $(a,b)=(1,0)$\\*
 & $\delta=0$: $(0,1,c,d)$ & $\B E_1(d,c,1,0)$ in the basis $e_2,e_1$\\*
 & $\delta=0$: $(0,b,c,0)$, $bc=1$ & $\B P_{10}^{b}$ $(b\neq1)$; $\B P_{12}$ $(b=1)$\\*
 & $\delta=0$: $(0,b,c,0)$, $bc\neq1$, $b\neq1$ & $\B E_1\big(\frac{(1-b)(1-c)}{1-bc},b,1,0\big)$ in the basis\qquad\qquad\qquad\qquad{ \qquad}\qquad $e_1$, $\frac{(1-c)e_1+(1-b)e_2}{1-bc}$\\ \hline
\end{longtable}}

\noindent Among these solutions, the identity \eqref{eq2} holds only for $\B A_1(0)$ with $\delta=1$; $\B A_3$; $\B B_2(0)$ with
$\delta=0$; $\B B_2(1)$; $\B D_1(0,0)$ and $\B D_1(1,0)$ with $\delta=0$; $\B D_2(0,0)$, $\B D_2(1,1)$ and $\B D_2(0,\delta^{-1})$; $\B D_2(\alpha,1)$
with $\delta=1$; $\B E_1(0,0,0,0)$, $\B E_1(0,1,0,1)$ and $\B E_1(1,0,1,0)$; $\B E_1(-1,0,0,-1)$ with $\delta=-1$; and $\B E_5(0)$
with $\delta=1$. This gives the list of Theorem \ref{thm11}.

It remains to describe the isomorphisms between the algebras obtained. For the classes $\mathbf A$--$\mathbf D$
they are given by the uniqueness part of the classification in the corrected version of \cite{b34} (for example,
$\B D_1(1,0)\cong\B D_1(0,0)$, since the parameters of $\B D_1$ are considered up to $(\alpha,\beta)\mapsto(1-\alpha+\beta,\beta)$).
An isomorphism of algebras of the class $\mathbf E$ maps nonzero idempotents to nonzero idempotents. The
algebras $\B E_5(\alpha)$ \textup(in particular, $\B P_{12}=\B E_5(1)$, $\B N_{07}^{1}$ and $\B P_{09}^{\delta/(2\delta-1),\delta}$\textup) have infinitely many
nonzero idempotents $xe_1+(1-x)e_2$, while each of the remaining algebras of the class $\mathbf E$ in the table
has two or three nonzero idempotents. Computing the multiplication tables $\B E_1(a,b,c,d)$ with respect
to all ordered pairs of different nonzero idempotents, one obtains exactly the identifications
$\B P_{10}^{\beta}\cong\B P_{10}^{1/\beta}$ and $\B P_{11}^{\alpha,\beta}\cong\B P_{11}^{\alpha,\beta'}$ described in the remark after Theorem \ref{thm19}, and no
other isomorphisms. \hfill$\Box$

\section{\texorpdfstring{$\delta$}{δ}-Novikov--Poisson algebras}\label{sec2}
In this section, we introduce $\delta$-Novikov--Poisson algebras, which generalize Novikov--Poisson
algebras. Also, we generalize several theorems from Novikov--Poisson algebras
\cite{b45,b46} to the context of $\delta$-Novikov--Poisson algebras.

\begin{definition}
Let $\delta$ be a fixed complex number. A $\delta$-Novikov--Poisson algebra $(N,\cdot,\circ)$ is
a vector space $N$ with two operations $\cdot$ and $\circ$ such that $(N,\cdot)$ is a commutative associative
algebra and $(N,\circ)$ is a $\delta$-Novikov algebra for which the following identities hold true:
\begin{eqnarray}
\label{eq6}(xy)\circ z=x(y\circ z),\\
\label{eq7}\delta(x\circ y)z-x\circ(yz)=\delta(y\circ x)z-y\circ(xz).
\end{eqnarray}
\end{definition}

\begin{example}
Let $(\A,\cdot)$ be a commutative associative algebra and let $\varphi$ be a $\delta$-derivation
of $\A$. Then $(\A,\cdot,\circ)$ is a $\delta$-Novikov--Poisson algebra under multiplication \eqref{eq3}.
\end{example}

\begin{lemma}\label{lem22}
Let $(N,\cdot,\circ)$ be a $\delta$-Novikov--Poisson algebra. For any fixed element $h\in N$, we
define a new operation ``$\times$'' on $N$ by
\[x\times y=x\circ y+hxy.\]
Then $(N,\cdot,\times)$ forms a $\delta$-Novikov--Poisson algebra.
\end{lemma}

\begin{proof}
Let us first prove that $(N,\times)$ forms a $\delta$-Novikov algebra. For $x,y,z\in N$, by \eqref{eq2} and
\eqref{eq6} we have
\begin{eqnarray*}
(x\times y)\times z&=&(x\circ y+hxy)\times z=(x\circ y)\circ z+(hxy)\circ z+h(x\circ y)z+h(hxy)z\\
&=&(x\circ z)\circ y+(hyx)\circ z+(hz)(x\circ y)+hhxyz\\
&=&(x\circ z)\circ y+h(x\circ z)y+(hzx)\circ y+hhxyz\\
&=&(x\times z)\times y.
\end{eqnarray*}

Now, by \eqref{eq1}
\begin{eqnarray*}
\lefteqn{\delta(x\times y)\times z-x\times(y\times z)\ =\ \delta(x\circ y+hxy)\times z-x\times(y\circ z+hyz)}\\
&=&\delta\left((x\circ y)\circ z+(hxy)\circ z+h(x\circ y)z+h(hxy)z\right)\\
&&\hfill-x\circ(y\circ z)-x\circ(hyz)-hx(y\circ z)-hx(hyz)\\
&=&\delta(x\circ y)\circ z-x\circ(y\circ z)+\delta(x\circ y)(hz)-x\circ(y(hz))+(\delta-1)((hx)y)\circ z+(\delta-1)hx(hyz)\\
&=&\delta(y\circ x)\circ z-y\circ(x\circ z)+\delta(y\circ x)(hz)-y\circ(x(hz))+(\delta-1)((hx)y)\circ z+(\delta-1)hx(hyz)\\
&=&\delta(y\times x)\times z-y\times(x\times z).
\end{eqnarray*}

Thus, $(N,\times)$ is a $\delta$-Novikov algebra.

Next, we prove the compatibility conditions for operations ``$\cdot$'' and ``$\times$''. For $x,y,z\in N$,
we have
\begin{longtable}{lclcl}
$(xy)\times z$&$=$&$(xy)\circ z+h(xy)z$&$=$&$ x(y\circ z)+x(hyz)$\\
&$=$&$x(y\circ z+hyz)$&$ =$&$ x(y\times z).$
\end{longtable}

Moreover,
\begin{longtable}{lclclcl}
$\delta(x\times y)z-x\times(yz)$&$=$&
$\delta(x\circ y+hxy)z-x\circ(yz)-hx(yz)$\\
&$=$&$\delta(x\circ y)z-x\circ(yz)+(\delta-1)hxyz$\\
&$=$&$\delta(y\circ x)z-y\circ(xz)+(\delta-1)hxyz$\\
&$=$&$\delta(y\circ x)z+\delta hxyz-y\circ(xz)-hyx\cdot z$\\
&$=$&$\delta(y\times x)z-y\times(xz).$
\end{longtable}

This completes the proof.
\end{proof}

\begin{theorem}\label{thm23}
Let $(N,\cdot,\circ)$ be a $\delta$-Novikov--Poisson algebra. For any fixed elements $q,h\in N$,
we define two new operations ``$\cdot_q$'' and ``$\times_h$'' on $N$ by
\[x\cdot_q y\ =\ qxy\ \ \text{and}\ \ x\times_h y\ =\ x\circ y+hxy.\]
Then $(N,\cdot_q,\times_h)$ forms a $\delta$-Novikov--Poisson algebra.
\end{theorem}

\begin{proof}
The proof follows from Lemma \ref{lem22} and the commutative associative property of
the product ``$\cdot$''.
\end{proof}

\begin{corollary}\label{cor24}
Let $\A$ be a commutative associative algebra and let $\varphi$ be a $\delta$-derivation of
$\A$. For fixed elements $q,h\in\A$, we define the operations ``$\cdot_q$'' and ``$\times_h$'' on $\A$ by
\[x\cdot_q y=qxy,\quad\ x\times_h y=x\varphi(y)+hxy.\]
Then $(\A,\cdot_q,\times_h)$ forms a $\delta$-Novikov--Poisson algebra.
\end{corollary}

The Kantor product of two multiplications defined on the same vector space gives
a way to construct a new multiplication satisfying some interesting properties. So, the
Kantor product of two multiplications of a Novikov--Poisson algebra gives a Novikov algebra
\cite{b23}. As we will see, a similar property can be extended to $\delta$-Novikov--Poisson
algebras.

\begin{definition}[see, \cite{b23}]
Let $u$ be a fixed vector of the vector space $V_n$ and $A,B$ be two
bilinear multiplications defined on $V_n$. Then $*$ is a new multiplication called the Kantor
product of $A$ and $B$ if
\[x*y=A(u,B(x,y))-B(A(u,x),y)-B(x,A(u,y)).\]
\end{definition}

\begin{theorem}\label{thm26}
Let $(N,\cdot,\circ)$ be a $\delta$-Novikov--Poisson algebra and let $*$ be the Kantor product
of $\cdot$ and $\circ$. Then $(N,*)$ is a $\delta$-Novikov algebra.
\end{theorem}

\begin{proof}
Firstly, we have
\begin{eqnarray*}
x*y&=&u(x\circ y)-(ux)\circ y-x\circ(uy)\ \ =\ \ -x\circ(uy).
\end{eqnarray*}
Hence,
\begin{eqnarray*}
(x*y)*z&=&(x\circ(uy))\circ(uz)\ \ =\ \ (x\circ(uz))\circ(uy)=(x*z)*y,
\end{eqnarray*}
and
\begin{longtable}{lclclcl}
$(x,y,z)^*_\delta$&$=$&$\delta(x*y)*z-x*(y*z)$\\
&$=$&$\delta(x\circ(uy))\circ(uz)-x\circ(u(y\circ(uz)))$\\
&$=$&$\delta(x\circ(uy))\circ(uz)-x\circ((uy)\circ(uz))$&$=$&$ (x,uy,uz)^\circ_\delta$&$=$&$(uy,x,uz)^\circ_\delta$\\
&$=$&$\delta((uy)\circ x)\circ(uz)-(uy)\circ(x\circ(uz))$\\
&$=$&$u(\delta(y\circ x)\circ(uz)-y\circ(x\circ(uz)))$&$=$&$ u(y,x,uz)^\circ_\delta$&$=$&$ u(x,y,uz)^\circ_\delta.$
\end{longtable}

The last observation gives $(x,y,z)^*_\delta=(y,x,z)^*_\delta$ and the theorem is proved.
\end{proof}

\begin{corollary}
Let $\A$ be a commutative associative algebra with a nonzero $\delta$-derivation
$\varphi$ and a fixed element $u\in\A$. Define a new multiplication $*$ by
\[x*y=x\varphi(uy).\]
Then $(\A,*)$ is a $\delta$-Novikov algebra.
\end{corollary}

\begin{proof}
The Kantor product in this case is $-x\varphi(uy)$, and the identities \eqref{eq1} and \eqref{eq2} are
invariant under the change of the multiplication $*\mapsto-*$.
\end{proof}

An important property of Poisson, transposed Poisson, anti-pre-Lie Poisson, and
Novikov--Poisson algebras is that they are closed under taking tensor products (see
\cite{b4,b35,b45})\footnote{On the other hand, $F$-manifold algebras and contact bracket algebras, in general, do not have this
property \cite{b37,b50}.}. We generalize this property to $\delta$-Novikov--Poisson algebras.

\begin{theorem}\label{thm28}
Let $(N_1,\cdot_1,\circ_1)$ and $(N_2,\cdot_2,\circ_2)$ be two $\delta$-Novikov--Poisson algebras. Define
operations on $N_1\otimes N_2$ as follows
\begin{eqnarray}
\label{eq8}(x_1\otimes x_2)\cdot(y_1\otimes y_2)&=&(x_1\cdot_1y_1)\otimes(x_2\cdot_2y_2),\\
\label{eq9}(x_1\otimes x_2)\circ(y_1\otimes y_2)&=&(x_1\circ_1y_1)\otimes(x_2\cdot_2y_2)+(x_1\cdot_1y_1)\otimes(x_2\circ_2y_2),
\end{eqnarray}
where $x_1,y_1\in N_1$ and $x_2,y_2\in N_2$. Then $(N_1\otimes N_2,\cdot,\circ)$ is a $\delta$-Novikov--Poisson algebra.
\end{theorem}

\begin{proof}
To simplify notation, the indices $1$ and $2$ in the operations $\cdot$ and $\circ$ will be omitted,
as their intended meaning will be clear from the context.

It is clear that $(N_1\otimes N_2,\cdot)$ is a commutative associative algebra. First, we prove that
$(N_1\otimes N_2,\circ)$ is a $\delta$-Novikov algebra. Let us consider $x=x_1\otimes x_2$, $y=y_1\otimes y_2$, and $z=z_1\otimes z_2$
where $x_1,y_1,z_1\in N_1$ and $x_2,y_2,z_2\in N_2$, then by \eqref{eq8} and \eqref{eq9} we have
\begin{eqnarray*}
\big(x\circ y\big)\circ z&=&\big((x_1\circ y_1)\otimes(x_2\cdot y_2)+(x_1\cdot y_1)\otimes(x_2\circ y_2)\big)\circ(z_1\otimes z_2)\\
&=&\big((x_1\circ y_1)\circ z_1\big)\otimes\big(x_2\cdot y_2\cdot z_2\big)+\big((x_1\circ y_1)\cdot z_1\big)\otimes\big((x_2\cdot y_2)\circ z_2\big)\\
&&\ \ +\big((x_1\cdot y_1)\circ z_1\big)\otimes\big((x_2\circ y_2)\cdot z_2\big)+\big(x_1\cdot y_1\cdot z_1\big)\otimes\big((x_2\circ y_2)\circ z_2\big)\\
&=&\big((x_1\circ z_1)\circ y_1\big)\otimes\big(x_2\cdot z_2\cdot y_2\big)+\big((x_1\cdot z_1)\circ y_1\big)\otimes\big((x_2\circ z_2)\cdot y_2\big)\\
&&\ \ +\big((x_1\circ z_1)\cdot y_1\big)\otimes\big((x_2\cdot z_2)\circ y_2\big)+\big((x_1\cdot z_1)\cdot y_1\big)\otimes\big((x_2\circ z_2)\circ y_2\big)\\
&=&\big(x\circ z\big)\circ y.
\end{eqnarray*}

Therefore, the identity \eqref{eq2} holds. Moreover,
\begin{eqnarray*}
x\circ\big(y\circ z\big)&=&\big(x_1\circ(y_1\circ z_1)\big)\otimes\big(x_2\cdot y_2\cdot z_2\big)+\big(x_1\cdot(y_1\circ z_1)\big)\otimes\big(x_2\circ(y_2\cdot z_2)\big)\\
&&+\big(x_1\circ(y_1\cdot z_1)\big)\otimes\big(x_2\cdot(y_2\circ z_2)\big)+\big(x_1\cdot y_1\cdot z_1\big)\otimes\big(x_2\circ(y_2\circ z_2)\big).
\end{eqnarray*}

Now, by \eqref{eq8} and \eqref{eq9} we have
\begin{eqnarray*}
\big(x,\ y,\ z\big)^\circ_\delta&=&(x_1,y_1,z_1)^\circ_\delta\otimes\big(x_2\cdot y_2\cdot z_2\big)+\big(x_1\cdot y_1\cdot z_1\big)\otimes(x_2,y_2,z_2)^\circ_\delta\\
&&+\big(x_1\cdot(y_1\circ z_1)\big)\otimes\big(\delta(x_2\circ y_2)\cdot z_2-x_2\circ(y_2\cdot z_2)\big)\\
&&\quad+\big(\delta(x_1\circ y_1)\cdot z_1-x_1\circ(y_1\cdot z_1)\big)\otimes\big(x_2\cdot(y_2\circ z_2)\big)\\
&=&(y_1,x_1,z_1)^\circ_\delta\otimes\big(y_2\cdot x_2\cdot z_2\big)+\big(y_1\cdot x_1\cdot z_1\big)\otimes(y_2,x_2,z_2)^\circ_\delta\\
&&+\big(y_1\cdot(x_1\circ z_1)\big)\otimes\big(\delta(y_2\circ x_2)\cdot z_2-y_2\circ(x_2\cdot z_2)\big)\\
&&\quad+\big(\delta(y_1\circ x_1)\cdot z_1-y_1\circ(x_1\cdot z_1)\big)\otimes\big(y_2\cdot(x_2\circ z_2)\big)\ \ =\ \ \big(y,\ x,\ z\big)^\circ_\delta.
\end{eqnarray*}

Here we used that $x_1\cdot(y_1\circ z_1)=(x_1\cdot y_1)\circ z_1=y_1\cdot(x_1\circ z_1)$ by \eqref{eq6}, while the remaining
factors are symmetric in $x$ and $y$ by \eqref{eq1} and \eqref{eq7}.

The condition \eqref{eq6} is straightforward to verify. Finally, by \eqref{eq8} and \eqref{eq9},
\begin{eqnarray*}
\delta\,(x\circ y)\cdot z-x\circ(y\cdot z)&=&\big(\delta(x_1\circ y_1)\cdot z_1-x_1\circ(y_1\cdot z_1)\big)\otimes(x_2\cdot y_2\cdot z_2)\\
&&+(x_1\cdot y_1\cdot z_1)\otimes\big(\delta(x_2\circ y_2)\cdot z_2-x_2\circ(y_2\cdot z_2)\big),
\end{eqnarray*}
and both brackets are symmetric in $x$ and $y$ by \eqref{eq7}; hence the compatibility condition \eqref{eq7} holds
in $N_1\otimes N_2$.

This completes the proof.
\end{proof}

\section{Commutator products of \texorpdfstring{$\delta$}{δ}-Novikov algebras}\label{sec3}
In this section, we consider $\delta$-Novikov algebras under the commutator product $[a,b]$. We
obtain a generalization of the well-known result that when $\delta=1$, Novikov algebras are
Lie under the commutator product. Additionally, it has been proven that when $\delta=0$,
bicommutative algebras are metabelian Lie under the commutator product \cite{b20}. We
establish that when $\delta\neq1$, every $\delta$-Novikov algebra is metabelian Lie-admissible.

First, we establish several lemmas that will be needed throughout the the proof of the
main result of this section.

\begin{lemma}\label{lem29}
Let $\delta\neq1$ and let $(N,\cdot)$ be a $\delta$-Novikov algebra, then
the following identity holds
\begin{eqnarray}\label{eq10}
\big(z(xy)\big)t&=&\big(z(xt)\big)y.
\end{eqnarray}
\end{lemma}

\begin{proof}
Let $x,y,z,t\in N$. Then it is clear that the following polynomial is an identity in $N$:
\begin{eqnarray*}
\big((x,z,t)_\delta-(z,x,t)_\delta\big)y&-&\big((x,z,y)_\delta-(z,x,y)_\delta\big)t\\
&+&(z,xy,t)_\delta-(xy,z,t)_\delta-(z,xt,y)_\delta+(xt,z,y)_\delta\ \ =\ \ 0.
\end{eqnarray*}

On the other hand,
\begin{eqnarray*}
\big((x,z,t)_\delta-(z,x,t)_\delta\big)y&-&\big((x,z,y)_\delta-(z,x,y)_\delta\big)t\\
&+&(z,xy,t)_\delta-(xy,z,t)_\delta-(z,xt,y)_\delta+(xt,z,y)_\delta\\
&=&\delta((xz)t)y-(x(zt))y-\delta((zx)t)y+(z(xt))y\\
&&\ \ -\delta((xz)y)t+(x(zy))t+\delta((zx)y)t-(z(xy))t\\
&&+\delta(z(xy))t-z((xy)t)-\delta((xy)z)t+(xy)(zt)\\
&&\ \ -\delta(z(xt))y+z((xt)y)+\delta((xt)z)y-(xt)(zy)\\
&\overset{(2)}{=}&(\delta-1)\big((z(xy))t-(z(xt))y\big).
\end{eqnarray*}

Since $\delta\neq1$, we have $(z(xy))t=(z(xt))y$. This completes the proof.
\end{proof}

\begin{lemma}\label{lem30}
Let $\delta\neq1$ and let $(N,\cdot)$ be a $\delta$-Novikov algebra, then
the following identity holds
\begin{eqnarray}\label{eq11}
[xy,zt]&=&\delta\big(((xy)z)t-((zx)y)t\big).
\end{eqnarray}
\end{lemma}

\begin{proof}
Let $x,y,z,t\in N$. Since $\delta\neq1$, by identities \eqref{eq2} and \eqref{eq10} we have
\begin{eqnarray*}
0&=&(x,z,t)_\delta y-(z,x,t)_\delta y\\
&=&\delta((xz)t)y-(x(zt))y-\delta((zx)t)y+(z(xt))y\\
&\overset{(2),(10)}{=}&\delta((xy)z)t-(xy)(zt)-\delta((zx)y)t+(zt)(xy).
\end{eqnarray*}

This completes the proof.
\end{proof}

We are now ready to prove the main result of this section.

\begin{lemma}\label{lem31}
Let $\delta\neq1$ and let $(N,\cdot)$ be a $\delta$-Novikov algebra. Then
$(N,[\cdot,\cdot])$ is a metabelian Lie algebra.
\end{lemma}

\begin{proof}
By Lemma \ref{lem7}, it remains to show that the metabelian identity holds in $N$. Then by
definition of the commutator product $[a,b]$ and the identities \eqref{eq2}, \eqref{eq10} and \eqref{eq11}, we have
\begin{eqnarray*}
[[x,y],[z,t]]&=&[xy,zt]-[yx,zt]-[xy,tz]+[yx,tz]\\
&=&\delta((xy)z)t-\delta((zx)y)t-\delta((yx)z)t+\delta((zx)y)t\\
&&\ \ -\delta((xy)z)t+\delta((tx)y)z+\delta((yx)z)t-\delta((tx)y)z\ \ =\ \ 0.
\end{eqnarray*}

Therefore, the lemma is proved.
\end{proof}

\begin{theorem}\label{thm32}
Every identity satisfied by the commutator product $([a,b]=a\cdot b-b\cdot a)$ in
every $\delta$-Novikov algebra with $\delta\neq1$ is a consequence of
anticommutativity, the Jacobi identity, and the metabelian identity.
\end{theorem}

\begin{proof}
By Lemma \ref{lem31}, we know that $(N,[\cdot,\cdot])$ is a metabelian Lie algebra. In \cite{b3}, it was
proved that every proper subvariety of the variety of metabelian Lie algebras is nilpotent. If there is an identity that does not follow from anticommutativity,
the Jacobi identity and the metabelian identity, then as we mentioned
by the result in \cite{b3} $(N,[\cdot,\cdot])$ must be nilpotent (for all $\delta$-Novikov algebras $N$, with a common
nilpotency index). Therefore, to obtain a contradiction, it is enough to show there
is a non-nilpotent metabelian Lie algebra obtained from a $\delta$-Novikov algebra. Let us consider
the algebra $\A$ in Example \ref{ex4}. Then it is clear that $[e_1,[e_1,[\ldots,[e_1,e_2]\ldots]]]=e_2$, thus
$(\A,[\cdot,\cdot])$ is not nilpotent (it is the $2$-dimensional non-abelian Lie algebra). Since
\[\A\in\Vz=\bigcap_\delta\Var_\delta,\]
we have the desired result.
\end{proof}

In \cite{b30}, the authors proved that every metabelian Lie algebra can be embedded into
some $2$-step left nilpotent bicommutative algebra under the commutator $[a,b]=ab-ba$.
More precisely, they proved that every metabelian Lie algebra can be embedded into an
algebra from the subvariety of the variety of bicommutative algebras with respect to the
commutator \cite[Corollary 2.5]{b30}.

Let $\mathcal W$ be the subvariety of the variety $\Vz$ defined by the following identities:
\[x(yz)=y(xz),\ \ (xy)z=0.\]

Passing to the opposite algebra in \cite[Corollary 2.5]{b30}, we obtain:

\begin{theorem}\label{thm33}
Every metabelian Lie algebra can be embedded into an algebra of $\mathcal W$ with
respect to the commutator product.
\end{theorem}

\begin{proof}
Let $f:L\to B$ be an embedding of a metabelian Lie algebra into a right commutative
algebra $B$ with $x(yz)=0$ \cite[Corollary 2.5]{b30}. The opposite algebra $B^{op}$ $(x*y=yx)$
satisfies $x*(y*z)=y*(x*z)$ and $(x*y)*z=0$, i.e. $B^{op}\in\mathcal W$, and its commutator is
$-[\cdot,\cdot]$; hence $-f:L\to B^{op}$ is an embedding with respect to the commutator. Note also
that $\mathcal W\subseteq\Vz$ (both sides of $(xy)z=(xz)y$ and $(xy)z=(yx)z$ vanish).
\end{proof}

Thus Theorem \ref{thm33}, which is a reformulation of \cite[Corollary 2.5]{b30}, together with the inclusions
$\mathcal W\subseteq\Vz\subseteq\Var_\delta$ immediately gives the following consequence.

\begin{corollary}
Every metabelian Lie algebra can be embedded into a suitable $\delta$-Novikov
algebra with respect to the commutator product.
\end{corollary}

\section{Relations to (transposed) \texorpdfstring{$\delta$}{δ}-Poisson algebras}\label{sec4}
In this section, we introduce the notions of $\delta$-Poisson algebras and transposed $\delta$-Poisson
algebras, and we demonstrate how they can be constructed using $\delta$-derivations. This
construction parallels the classical approach to Poisson algebras via derivations.

One of the standard examples of Poisson algebras is:

\begin{example}
Consider the polynomial algebra $H_n=\mathbb C[x_1,\ldots,x_n,y_1,\ldots,y_n]$, equipped with the Poisson bracket defined by
\[
\{f,g\}=\sum_{i=1}^n\left(\frac{\partial f}{\partial x_i}\frac{\partial g}{\partial y_i}-\frac{\partial f}{\partial y_i}\frac{\partial g}{\partial x_i}\right),\quad\forall f,g\in H_n,
\]
where $\frac{\partial}{\partial z}$ denotes the partial derivative with respect to the variable $z\in
\{x_1,\ldots,x_n,y_1,\ldots,y_n\}$. Then $H_n$ is a Poisson algebra because the multiplication $\{\cdot,\cdot\}$
satisfies the Lie algebra identities and the Leibniz rule with respect to the commutative
associative multiplication in $H_n$.
\end{example}

More generally, we have the following:

\begin{example}\label{ex36}
Let $(\A,\cdot)$ be a commutative associative algebra, and let $D_1,D_2$ be commuting
derivations of $\A$, that is, $D_1D_2=D_2D_1$. Define a new multiplication on $\A$ by
\[\{x,y\}=D_1(x)D_2(y)-D_1(y)D_2(x).\]
Then $(\A,\{\cdot,\cdot\})$ is a Lie algebra, and $(\A,\cdot,\{\cdot,\cdot\})$ is a Poisson algebra.
\end{example}

Now, we define the generalization of these notions by introducing the definition of
$\delta$-Poisson algebras, which first appeared in \cite{b1}.

\begin{definition}[see \cite{b1}]
Let $\delta$ be a fixed complex number. An algebra $(\mathrm P,\cdot,\{\cdot,\cdot\})$ is called a
$\delta$-Poisson algebra\footnote{When $\delta=1$, we recover the usual definition of Poisson algebras.} if $(\mathrm P,\cdot)$ is a commutative associative algebra, $(\mathrm P,\{\cdot,\cdot\})$ is a Lie algebra
and the following identity holds:
\begin{eqnarray}\label{eq12}
\{xy,z\}=\delta\big(x\{y,z\}+\{x,z\}y\big).
\end{eqnarray}
\end{definition}

We can construct $\delta$-Poisson algebras using $\delta$-derivations.

\begin{theorem}\label{thm38}
Let $\A$ be a commutative associative algebra and let $\varphi_1$ and $\varphi_2$ be commuting
$\delta$-derivations of $\A$. Define
\[\lb x,y\rb=\varphi_1(x)\varphi_2(y)-\varphi_2(x)\varphi_1(y).\]
Then $(\A,\cdot,\lb\cdot,\cdot\rb)$ is a $\delta$-Poisson algebra.
\end{theorem}

\begin{proof}
First, we note that $(\A,\lb\cdot,\cdot\rb)$ is a Lie algebra. This follows from
\begin{eqnarray*}
\lb\lb x,y\rb,z\rb&=&\delta\Big(\ \begin{array}[t]{l}
\varphi_1^2(x)\varphi_2(y)\varphi_2(z)+\varphi_1(x)\varphi_1\varphi_2(y)\varphi_2(z)-\\
\varphi_2\varphi_1(x)\varphi_2(y)\varphi_1(z)-\varphi_1(x)\varphi_2^2(y)\varphi_1(z)-\\
\varphi_1\varphi_2(x)\varphi_1(y)\varphi_2(z)-\varphi_2(x)\varphi_1^2(y)\varphi_2(z)+\\
\varphi_2^2(x)\varphi_1(y)\varphi_1(z)+\varphi_2(x)\varphi_2\varphi_1(y)\varphi_1(z)\quad\Big).
\end{array}
\end{eqnarray*}

The latter implies that
\begin{eqnarray*}
\lb\lb x,y\rb,z\rb+\lb\lb y,z\rb,x\rb+\lb\lb z,x\rb,y\rb&=&0.
\end{eqnarray*}

By definition of the new multiplication $\lb\cdot,\cdot\rb$ and given that $\varphi_1$ and $\varphi_2$ are
$\delta$-derivations we have
\begin{eqnarray*}
\lb x,y\cdot z\rb&=&\varphi_1(x)\varphi_2(y\cdot z)-\varphi_2(x)\varphi_1(y\cdot z)\\
&=&\varphi_1(x)\cdot\delta\big(\varphi_2(y)\cdot z+y\cdot\varphi_2(z)\big)-\varphi_2(x)\cdot\delta\big(\varphi_1(y)\cdot z+y\cdot\varphi_1(z)\big)\\
&=&\delta\big(\varphi_1(x)\varphi_2(y)\cdot z+\varphi_1(x)\cdot y\cdot\varphi_2(z)-\varphi_2(x)\varphi_1(y)\cdot z-\varphi_2(x)\cdot y\cdot\varphi_1(z)\big)\\
&=&\delta\big(\big(\varphi_1(x)\varphi_2(y)-\varphi_2(x)\varphi_1(y)\big)\cdot z+y\cdot\big(\varphi_1(x)\varphi_2(z)-\varphi_2(x)\varphi_1(z)\big)\big)\\
&=&\delta\big(\lb x,y\rb\cdot z+y\cdot\lb x,z\rb\big).
\end{eqnarray*}

By anticommutativity, this is equivalent to \eqref{eq12}.
\end{proof}

\begin{definition}[see, \cite{b1,b13}]
Let $\delta$ be a fixed complex number. An algebra $(\mathrm P,\cdot,\{\cdot,\cdot\})$ is
called a transposed $\delta$-Poisson algebra\footnote{When $\delta=2$, this definition coincides with the standard notion of transposed Poisson algebras, as
introduced in \cite{b4}; for $\delta=-1$ one obtains transposed anti-Poisson algebras.} if $(\mathrm P,\cdot)$ is a commutative associative algebra,
$(\mathrm P,\{\cdot,\cdot\})$ is a Lie algebra and the following identity holds:
\begin{eqnarray}\label{eq13}
\delta x\{y,z\}=\{xy,z\}+\{y,xz\}.
\end{eqnarray}
\end{definition}

A similar example to Example \ref{ex36}, leading to transposed Poisson algebras, can be
defined by using only one derivation.

\begin{example}[see, \cite{b4}]
Let $(\A,\cdot)$ be a commutative associative algebra, and let $D$ be a
derivation of $\A$. Define a new multiplication on $\A$ by
\[\{x,y\}=xD(y)-D(x)y.\]
Then $(\A,\{\cdot,\cdot\})$ is a Lie algebra, and $(\A,\cdot,\{\cdot,\cdot\})$ forms a transposed Poisson algebra.
\end{example}

We establish a connection between $\delta$-Novikov--Poisson algebras and transposed $(\delta+1)$-Poisson
algebras.

\begin{theorem}\label{thm41}
Let $(N,\cdot,\circ)$ be a $\delta$-Novikov--Poisson algebra. Define
\[[x,y]=x\circ y-y\circ x.\]
Then $(N,\cdot,[\cdot,\cdot])$ is a transposed $(\delta+1)$-Poisson algebra.
\end{theorem}

\begin{proof}
From Lemma \ref{lem7}, we know that $(N,[\cdot,\cdot])$ is a Lie algebra. It remains to verify the
transposed $(\delta+1)$-Poisson identity \eqref{eq13}. For all $a,b,c\in N$, by \eqref{eq6} and \eqref{eq7} we have
\begin{eqnarray*}
[ac,b]+[a,bc]&=&(ac)\circ b-b\circ(ac)+a\circ(bc)-(bc)\circ a\\
&=&(ac)\circ b-\delta(b\circ a)c+\delta(a\circ b)c-(bc)\circ a\\
&=&c(a\circ b)-\delta(b\circ a)c+\delta(a\circ b)c-c(b\circ a)\\
&=&(\delta+1)(a\circ b)c-(\delta+1)(b\circ a)c\qquad\qquad\ \ =\ \ (\delta+1)c[a,b].
\end{eqnarray*}

Therefore, $(N,\cdot,[\cdot,\cdot])$ is a transposed $(\delta+1)$-Poisson algebra.
\end{proof}

As a corollary, we can construct transposed $(\delta+1)$-Poisson algebras using
$\delta$-derivations:

\begin{corollary}\label{cor42}
Let $\A$ be a commutative associative algebra and let $\varphi$ be a nontrivial
$\delta$-derivation of $\A$. Define
\[\lb a,b\rb=a\varphi(b)-\varphi(a)b.\]
Then $(\A,\cdot,\lb\cdot,\cdot\rb)$ is a transposed $(\delta+1)$-Poisson algebra.
\end{corollary}

\begin{proof}
Define the operation $a\circ b=a\varphi(b)$ for all $a,b\in\A$. Since $\varphi$ is a $\delta$-derivation, $(\A,\circ)$
is a $\delta$-Novikov algebra and $(\A,\cdot,\circ)$ is a $\delta$-Novikov--Poisson algebra. From the previous
theorem, we have $(\A,\cdot,[\cdot,\cdot])$ is a transposed $(\delta+1)$-Poisson algebra.
\end{proof}

It is known that each Novikov algebra can be embedded in a suitable commutative
associative algebra with derivation under the multiplication given by \eqref{eq3} \cite{b8}. A similar
problem has appeared for transposed $\delta$-Poisson algebras.

\noindent\textbf{Question}. Can a nontrivial transposed $(\delta+1)$-Poisson algebra be embedded in a
suitable $\delta$-Novikov--Poisson algebra under the commutator product?

Proposition \ref{prop43} below gives a negative answer for every value of $\delta$.

\begin{proposition}\label{prop43}
Let $P=\mathfrak{sl}_3\oplus\mathbb Cc$, where $\mathfrak{sl}_3=\mathfrak{sl}_3(\mathbb C)$, with
the Lie bracket of $\mathfrak{sl}_3$ \textup($c$ is central\textup) and the commutative associative multiplication $c\cdot c=c$
\textup(all products involving elements of $\mathfrak{sl}_3$ are zero\textup). Then, for every $\delta$, $P$ is a transposed
$(\delta+1)$-Poisson algebra with both multiplications nonzero, and $P$ cannot be embedded into a $\delta$-Novikov
algebra \textup(in particular, into a $\delta$-Novikov--Poisson algebra\textup) under the commutator product.
\end{proposition}

\begin{proof}
For all $x,y,z\in P$ we have $x\cdot\{y,z\}=0$, since $\{y,z\}\in\mathfrak{sl}_3$, and $\{x\cdot y,z\}=\{y,x\cdot z\}=0$,
since $x\cdot y$, $x\cdot z\in\mathbb Cc$; hence \eqref{eq13} holds for every value of the parameter. Suppose that $P$ is
embedded into a $\delta$-Novikov algebra $N$ under the commutator product; then $(P,\{\cdot,\cdot\})$ is a Lie subalgebra
of $(N,[\cdot,\cdot])$.

Let $\delta\neq1$. By Lemma \ref{lem31}, every Lie algebra embedded into a $\delta$-Novikov algebra under the commutator
product is metabelian. On the other hand, for $h=E_{11}-E_{22}$, $e=E_{12}$ and $f=E_{21}$ we have $\{\{h,e\},\{h,f\}\}=-4h\neq0$,
so $(P,\{\cdot,\cdot\})$ is not metabelian. \textup(For $\delta\neq1$ the same argument shows that the algebra $\mathfrak{sl}_2\oplus\mathbb Cc$ with
$c\cdot c=c$ cannot be embedded into a $\delta$-Novikov algebra.\textup)

Let $\delta=1$. The commutator algebra of every Novikov algebra satisfies the standard identity of degree $5$
\begin{equation}\label{eq16}\tag{S}
S(x_1,x_2,x_3,x_4;z):=\sum_{\sigma\in S_4}\mathrm{sgn}(\sigma)\,[x_{\sigma(1)},[x_{\sigma(2)},[x_{\sigma(3)},[x_{\sigma(4)},z]]]]=0
\end{equation}
\textup(see, for example, \cite[p.~38]{b17a}\textup). For the convenience of the reader, we recall the argument. By \cite{b8},
every Novikov algebra can be embedded into a commutative associative algebra $\A$ with a
derivation $D$ under the multiplication $a\circ b=aD(b)$, so it is enough to verify \eqref{eq16} for the commutator
$[a,b]=aD(b)-bD(a)$ on $\A$. By the Leibniz rule, $[x_1,[x_2,[x_3,[x_4,z]]]]$ is a linear combination of
the products $D^{r_1}(x_1)D^{r_2}(x_2)D^{r_3}(x_3)D^{r_4}(x_4)D^{r_5}(z)$ with $r_1+\cdots+r_5=4$. Since $\A$ is commutative,
for fixed $r_1,\ldots,r_5$ the alternating sum over $\sigma\in S_4$ of the products $D^{r_1}(x_{\sigma(1)})\cdots D^{r_4}(x_{\sigma(4)})$ is
the determinant of the matrix $\big(D^{r_i}(x_j)\big)_{i,j=1}^4$, which vanishes if two of the numbers $r_1,\ldots,r_4$
coincide. Four different nonnegative integers have sum at least $0+1+2+3=6>4$; therefore, $S$ vanishes
identically on $\A$. On the other hand, a direct computation with the matrix units $E_{ij}$ gives
\[S(E_{12},E_{21},E_{23},E_{32};E_{13})=4E_{13}\neq0,\]
so $(P,\{\cdot,\cdot\})$ does not satisfy \eqref{eq16}. This contradiction proves the proposition.
\end{proof}

\begin{definition}
Let $\delta$ be a fixed complex number. An algebra $(\mathrm G,\circ,\{\cdot,\cdot\})$ is called a
$\delta$-Gelfand--Dorfman algebra\,\footnote{When $\delta=1$, this definition coincides with the standard notion of Gelfand--Dorfman algebras (see, for
example, \cite{b42}).} if $(\mathrm G,\circ)$ is a $\delta$-Novikov algebra, $(\mathrm G,\{\cdot,\cdot\})$ is a Lie algebra and
the following identity holds:
\begin{eqnarray}\label{eq14}
\{a,b\circ c\}-\{c,b\circ a\}+\delta\{b,a\}\circ c-\delta\{b,c\}\circ a-b\circ\{a,c\}=0.
\end{eqnarray}
\end{definition}

If $\delta=0$, then the identity \eqref{eq14} coincides with the identity \eqref{eq13} of transposed $1$-Poisson
algebras, provided that the multiplication $\circ$ is commutative.

\begin{proposition}
Let $(N,\circ)$ be a $\delta$-Novikov algebra. Define
\[\{a,b\}=a\circ b-b\circ a.\]
Then $(N,\circ,\{\cdot,\cdot\})$ is a $\delta$-Gelfand--Dorfman algebra.
\end{proposition}

\begin{proof}
From Lemma \ref{lem7}, we know that $(N,[\cdot,\cdot])$ is a Lie algebra. It remains to verify the
identity \eqref{eq14}. For all $a,b,c\in N$, by \eqref{eq1} and \eqref{eq2} we have
\[
\begin{array}{c}
\{a,b\circ c\}-\{c,b\circ a\}+\delta\{b,a\}\circ c-\delta\{b,c\}\circ a-b\circ\{a,c\}=\\[6pt]
(ba)c-(bc)a+(b,a,c)_\delta-(a,b,c)_\delta+(c,b,a)_\delta-(b,c,a)_\delta=0.
\end{array}
\]
\end{proof}

In \cite{b42}, using computer algebra methods, it was shown that the variety of transposed
Poisson algebras coincides with the variety of Gelfand--Dorfman algebras in which the
Novikov multiplication is commutative. We give a simple proof of this result and generalize
it to the varieties of commutative $\delta$-Gelfand--Dorfman algebras and transposed
$(\delta+1)$-Poisson algebras.

A $\delta$-Gelfand--Dorfman algebra $(\mathrm G,\circ,\{\cdot,\cdot\})$ is called commutative if the multiplication $\circ$
is commutative; then it is also associative, and we denote it by $\cdot$.

Consider the following polynomials:
\begin{eqnarray*}
GD(x,y,z)&=&\{x,yz\}-\{z,yx\}+\delta\{y,x\}z-\delta\{y,z\}x-y\{x,z\},\\
TP(x,y,z)&=&(\delta+1)x\{y,z\}+\{z,xy\}-\{y,xz\}.
\end{eqnarray*}

The following statement shows that the result also holds for $\delta$-Gelfand--Dorfman algebras
and transposed $(\delta+1)$-Poisson algebras.

\begin{theorem}\label{thm46}
A commutative $\delta$-Gelfand--Dorfman algebra $(\mathrm G,\cdot,\{\cdot,\cdot\})$ is a transposed $(\delta+
1)$-Poisson algebra if $\delta\neq\frac12$. Conversely, a transposed $(\delta+1)$-Poisson algebra $(\mathrm G,\cdot,\{\cdot,\cdot\})$ is a
commutative $\delta$-Gelfand--Dorfman algebra if $\delta\neq-1$.
\end{theorem}

\begin{proof}
Since $(\mathrm G,\cdot)$ is commutative associative algebra and $(\mathrm G,\{\cdot,\cdot\})$ is a Lie algebra (recall
that a commutative $\delta$-Novikov algebra is associative, since $x(yz)=(yz)x=(yx)z=(xy)z$,
and, conversely, a commutative associative algebra satisfies \eqref{eq1} and \eqref{eq2}), it remains to
show that the compatibility conditions \eqref{eq13} and \eqref{eq14} are equivalent. We have
\begin{eqnarray*}
TP(a,b,c)&=&\tfrac{1}{2\delta-1}\big(\delta GD(c,b,a)-\delta GD(b,c,a)-(1-\delta)GD(c,a,b)\big),\\
GD(a,b,c)&=&\tfrac{1}{\delta+1}\big(\delta TP(a,c,b)+\delta TP(c,b,a)+TP(b,c,a)\big).
\end{eqnarray*}

This completes the proof.
\end{proof}

\begin{remarkn}
The restrictions on $\delta$ given in Theorem \ref{thm46} are important due to the following
examples:
\begin{enumerate}
\item[{\rm 1.}] The following algebra gives an example of a transposed $0$-Poisson algebra, that is not
a $(-1)$-Gelfand--Dorfman algebra:
\[
\begin{array}{lcll}
\mathfrak A_1&:&e_1\cdot e_1=e_1&\{e_2,e_3\}=e_1
\end{array}
\]
\item[{\rm 2.}] The following algebra gives an example of a $\frac12$-Gelfand--Dorfman algebra, that is not a
transposed $\frac32$-Poisson algebra:
\[
\begin{array}{lclll}
\mathfrak A_2&:&e_1\cdot e_1=e_4&e_2\cdot e_2=e_4&e_3\cdot e_3=e_4\\
&&\{e_1,e_2\}=-e_3&\{e_1,e_3\}=e_2&\{e_2,e_3\}=-e_1
\end{array}
\]
\end{enumerate}
\end{remarkn}

\section{Operad of the variety of \texorpdfstring{$\delta$}{δ}-Novikov algebras}\label{sec5}
In this section, we prove that the operad associated with the variety of $\delta$-Novikov algebras
is not Koszul. Let $\delta$-$N$ denote the operad governing the variety of $\delta$-Novikov
algebras. Note that, when $\delta=1$, the operad $1$-$N$ coincides with the Novikov operad, and
when $\delta=0$, the operad $0$-$N$ coincides with the bicommutative operad. It has been proven
that both the Novikov and bicommutative operads are not Koszul \cite{b18,b21}. We extend
these results by proving that the operad $\delta$-$N$ is not Koszul for every value of $\delta$.

The proof of the main result of this section is based on the following lemma.

\begin{lemma}\label{lem47}
Let $\delta\neq1$, and let $x_1,x_2,x_3,x_4$ be free generators of the free $\delta$-Novikov algebra. Then every multilinear monomial of degree $4$ in $x_1,x_2,x_3,x_4$ is a linear combination of the
following $14$ monomials:
\[
\begin{array}{ll}
U_i=((x_ix_j)x_k)x_l, & \{j,k,l\}=\{1,2,3,4\}\setminus\{i\},\ j<k<l\quad(4\text{ monomials});\\[2pt]
V_{ij|kl}=(x_i(x_jx_k))x_l, & \{i,j,k,l\}=\{1,2,3,4\},\ i<j,\ k<l\quad(6\text{ monomials});\\[2pt]
W_l=x_i(x_j(x_kx_l)), & \{i,j,k\}=\{1,2,3,4\}\setminus\{l\},\ i<j<k\quad(4\text{ monomials}).
\end{array}
\]
In particular, the dimension of the multilinear component of degree $4$ of the free $\delta$-Novikov algebra
is at most $14$.
\end{lemma}

\begin{proof}
Below $x,y,z,t$ denote the generators $x_1,x_2,x_3,x_4$ taken in an arbitrary order. Every multilinear monomial
of degree $4$ has one of the forms $((xy)z)t$, $(x(yz))t$, $(xy)(zt)$, $x((yz)t)$ and $x(y(zt))$. We use \eqref{eq2},
the identity \eqref{eq10} of Lemma \ref{lem29} \textup(here $\delta\neq1$ is needed\textup) and the identity \eqref{eq1} in the form
\[a(xy)=\delta(ax)y-\delta(xa)y+x(ay).\]

(a) By \eqref{eq2}, $((xy)z)t=((xz)y)t$ and $((xy)z)t=((xy)t)z$. Hence $((xy)z)t$ is invariant under all
permutations of $y,z,t$, i.e. $((xy)z)t=U_i$, where $x=x_i$.

(b) By \eqref{eq2}, $(xy)(zt)=(x(zt))y$, so it is enough to consider the monomials $(x(yz))t$. By \eqref{eq10},
$(x(yz))t=(x(yt))z$. By \eqref{eq1}, $(x(yz))t-(y(xz))t=\delta((xy)z)t-\delta((yx)z)t$, which is a linear combination
of the monomials $U_i$ by (a). Hence, modulo the span of $U_1,\ldots,U_4$, the monomial $(x(yz))t$ depends only on
the unordered pairs $\{x,y\}$ and $\{z,t\}$, i.e. it is congruent to $V_{ij|kl}$, where $\{x,y\}=\{x_i,x_j\}$ and
$\{z,t\}=\{x_k,x_l\}$.

(c) By \eqref{eq1}, $x((yz)t)=\delta(x(yz))t-\delta((yz)x)t+(yz)(xt)$, which is a linear combination of monomials
considered in (a) and (b).

(d) By \eqref{eq1}, 
\begin{center}$x(y(zt))-y(x(zt))=\delta(xy)(zt)-\delta(yx)(zt)$ and $x(y(zt))-x(z(yt))=\delta x((yz)t)-\delta x((zy)t)$,\end{center}
and by (a)--(c) the right-hand sides are linear combinations of the monomials $U_i$ and $V_{ij|kl}$. Hence,
modulo the span of these monomials, $x(y(zt))$ is invariant under all permutations of $x,y,z$, i.e. it is
congruent to $W_l$, where $t=x_l$.
\end{proof}

\begin{theorem}\label{thm48}
Let $\delta$-$N$ be the operad governing  the variety of $\delta$-Novikov algebras. Then,
the operad $\delta$-$N$ is not Koszul.
\end{theorem}

\begin{proof}
Let us construct multilinear basis elements of degree $3$ for the free $\delta$-Novikov
algebra. Below, we give a presentation of $6$ non-basis elements of degree $3$ as a linear
combination of the base elements of degree $3$:

\begin{center}
$(a\circ c)\circ b\ =\ (a\circ b)\circ c,\qquad\ \ (b\circ c)\circ a\ =\ (b\circ a)\circ c,\qquad\ \ (c\circ b)\circ a\ =\ (c\circ a)\circ b,$
\end{center}
\[\begin{array}{rcl}
b\circ(a\circ c)&=&\delta(b\circ a)\circ c-\delta(a\circ b)\circ c+a\circ(b\circ c),\\
c\circ(a\circ b)&=&\delta(c\circ a)\circ b-\delta(a\circ c)\circ b+a\circ(c\circ b),\\
c\circ(b\circ a)&=&\delta(c\circ b)\circ a-\delta(b\circ c)\circ a+b\circ(c\circ a).
\end{array}
\]

Following the approach in \cite{b28}, we compute the dual operad $\delta$-$N^!$, where the operad
$\delta$-$N$ governs the variety of $\delta$-Novikov algebras. Then,
\begin{eqnarray*}
\lefteqn{[[a\otimes x,b\otimes y],c\otimes z]\ +\ [[b\otimes y,c\otimes z],a\otimes x]\ +\ [[c\otimes z,a\otimes x],b\otimes y]}\\
&=&((a\circ b)\circ c)\otimes((x\bullet y)\bullet z)-(c\circ(a\circ b))\otimes(z\bullet(x\bullet y))\\
&&-((b\circ a)\circ c)\otimes((y\bullet x)\bullet z)+(c\circ(b\circ a))\otimes(z\bullet(y\bullet x))\\
&&+((b\circ c)\circ a)\otimes((y\bullet z)\bullet x)-(a\circ(b\circ c))\otimes(x\bullet(y\bullet z))\\
&&-((c\circ b)\circ a)\otimes((z\bullet y)\bullet x)+(a\circ(c\circ b))\otimes(x\bullet(z\bullet y))\\
&&+((c\circ a)\circ b)\otimes((z\bullet x)\bullet y)-(b\circ(c\circ a))\otimes(y\bullet(z\bullet x))\\
&&-((a\circ c)\circ b)\otimes((x\bullet z)\bullet y)+(b\circ(a\circ c))\otimes(y\bullet(x\bullet z))\ =\\
\lefteqn{\Big(\text{ by using the above presentation of the 6 non-basis elements of degree 3}}\\
&&\hfill\text{as a linear combination of basis elements, we have }\Big)\\
&=&((a\circ b)\circ c)\otimes((x\bullet y)\bullet z-\delta y\bullet(x\bullet z)-(x\bullet z)\bullet y+\delta z\bullet(x\bullet y))\\
&&+((b\circ a)\circ c)\otimes(-(y\bullet x)\bullet z+\delta y\bullet(x\bullet z)+(y\bullet z)\bullet x-\delta z\bullet(y\bullet x))\\
&&+((c\circ a)\circ b)\otimes(-(z\bullet y)\bullet x+\delta z\bullet(y\bullet x)+(z\bullet x)\bullet y-\delta z\bullet(x\bullet y))\\
&&+(a\circ(b\circ c))\otimes(-x\bullet(y\bullet z)+y\bullet(x\bullet z))\\
&&+(a\circ(c\circ b))\otimes(x\bullet(z\bullet y)-z\bullet(x\bullet y))\\
&&+(b\circ(c\circ a))\otimes(-y\bullet(z\bullet x)+z\bullet(y\bullet x)).
\end{eqnarray*}

Therefore, the Lie-admissibility condition gives us defining identities for the dual
operad $\delta$-$N^!$, which are equivalent to the following:
\begin{eqnarray*}
(x\bullet y)\bullet z-\delta x\bullet(y\bullet z)&=&(x\bullet z)\bullet y-\delta x\bullet(z\bullet y),\\
x\bullet(y\bullet z)&=&y\bullet(x\bullet z).
\end{eqnarray*}

Thus, the Koszul dual of the operad of left $\delta$-Novikov algebras is the operad of right
$\delta$-Novikov algebras.

Moreover, the dimensions of the multilinear components of the free left $\delta$-Novikov
and free right $\delta$-Novikov algebras are equal. Specifically, for degrees $1$, $2$ and $3$ the
dimensions of the multilinear parts are $1$, $2$, and $6$, respectively, for every $\delta$ \textup(in degree $3$, the
above presentation of the $6$ non-basis elements shows that the dimension is at most $6$; on the other hand,
the defining relations in degree $3$ are obtained from \eqref{eq1} and \eqref{eq2} by permutations of the variables, and,
since the left side of \eqref{eq1} is skew-symmetric in $x,y$ and the polynomial $(xy)z-(xz)y$ is skew-symmetric
in $y,z$, they span a subspace of dimension at most $3+3=6$ of the $12$-dimensional space of multilinear
monomials\textup). Let us
denote by $\beta=\beta(\delta)$ and $\alpha=\alpha(\delta)$ the dimension of the multilinear part in degrees $4$ and $5$,
respectively. Therefore, the generating functions (Hilbert series) of the $\delta$-Novikov operad
and its Koszul dual are:
\[
H(x)=H^!(x)=-x+x^2-x^3+\frac{\beta x^4}{24}-\frac{\alpha x^5}{120}+O(x^6).
\]
We observe that the composition of the Hilbert series satisfies:
\[
H\big(H^!(x)\big)=x+\frac{240-15\beta+\alpha}{60}x^5+O(x^6).
\]

On the other hand, by Lemma \ref{lem47}, $\beta\leq14$ for $\delta\neq1$ \textup(for $\delta=1$ one has $\beta=20$, and the
Novikov operad is not Koszul by \cite{b18}\textup). Additionally, from the equation $240-15\beta+\alpha=0$,
we see that if $\beta<16$, then $\alpha$ must be negative. However, since $\alpha$ is nonnegative, it follows
that the coefficient of $x^5$ in $H\big(H^!(x)\big)$ is nonzero. Consequently, $H\big(H^!(x)\big)\neq x$. According to
the results of Ginzburg and Kapranov in \cite{b28}, no $\delta$-Novikov operad is Koszul. Therefore,
the statement of the theorem is proved.
\end{proof}

\begin{remark}
In fact, $\beta=14$ for every $\delta\neq1$, i.e. the $14$ monomials of Lemma \ref{lem47} form a basis of the multilinear
component of degree $4$. Indeed, this component is the quotient of the $120$-dimensional space spanned
by the multilinear nonassociative monomials of degree $4$ by the subspace spanned by the $240$ polynomials
\[
\begin{array}{c}
f(x_{\sigma(1)}x_{\sigma(2)},x_{\sigma(3)},x_{\sigma(4)}),\quad f(x_{\sigma(3)},x_{\sigma(1)}x_{\sigma(2)},x_{\sigma(4)}),\quad f(x_{\sigma(3)},x_{\sigma(4)},x_{\sigma(1)}x_{\sigma(2)}),\\[2pt]
f(x_{\sigma(1)},x_{\sigma(2)},x_{\sigma(3)})x_{\sigma(4)},\quad x_{\sigma(4)}f(x_{\sigma(1)},x_{\sigma(2)},x_{\sigma(3)}),
\end{array}
\]
where $\sigma\in S_4$ and $f$ runs over the
left sides of \eqref{eq1} and \eqref{eq2} written as $f(x,y,z)=0$. The coefficients of these polynomials are
polynomials of degree at most $1$ in $\delta$, and a direct computation shows that the corresponding
$240\times120$ matrix has rank $106$ over the field $\mathbb C(\delta)$ of rational functions. Since the rank does
not increase under specialization of $\delta$, we have $\beta\geq120-106=14$ for every $\delta$, and hence $\beta=14$
for $\delta\neq1$ by Lemma \ref{lem47}. For $\delta=1$ the rank is $100$, in accordance with $\beta=20$ \cite{b22}.
\end{remark}

\section{Remarks on solvability and nilpotency of \texorpdfstring{$\delta$}{δ}-Novikov algebras}\label{sec6}
In this section, we summarize some common properties of Novikov algebras and $\delta$-Novikov
algebras. Most of the lemmas and theorems presented in \cite[Section 3]{b44} for
Novikov algebras hold analogously for $\delta$-Novikov algebras with $\delta\neq0$ due to the structural
similarities and shared algebraic properties. Thus, the results for solvability
and nilpotency apply without modification in the $\delta$-Novikov case with $\delta\neq0$; for $\delta=0$,
Lemma \ref{lem51} and Theorem \ref{thm52} fail (see the Remark at the end of this section). For the convenience
of the reader, we sketch the proofs of Lemma \ref{lem50}, Lemma \ref{lem51} and Theorem \ref{thm52}.

The following theorem states that idempotent nonzero minimal ideals are simple,
which is analogous to the case for Novikov algebras \cite{b44}.

\begin{theorem}\label{thm49}
If $I$ is a minimal ideal of a $\delta$-Novikov algebra $N$, then either $I^2=(0)$ or $I$
is a simple algebra.
\end{theorem}

\begin{proof}
For $\delta=1$ this is the theorem of Shestakov and Zhang \cite{b44}; let $\delta\neq1$. The argument
below works in arbitrary dimension and uses neither $\delta\neq0$ nor
$\delta\neq-1$. Recall that \eqref{eq1} can be written as
\begin{eqnarray}\label{eq15}
a(xy)&=&\delta(ax)y-\delta(xa)y+x(ay).
\end{eqnarray}
By Lemma \ref{lem7} (1), $I^2$ is an ideal of $N$, hence $I^2=0$ or $I^2=I$. Let $I^2=I$ and let $J\neq0$ be an
ideal of the algebra $I$; we show that $J=I$. Below $a,b,i\in I$, $j\in J$, $n\in N$.

\textit{Step 1: $IJ$ is an ideal of $N$.} By \eqref{eq2} and \eqref{eq15},
\begin{align*}
(ab)(nj)&=(a(nj))b=(n(aj))b+\delta((an)j)b-\delta((na)j)b\\
&=(nb)(aj)+\delta((an)b)j-\delta((na)b)j\in IJ.
\end{align*}
Since $I=I^2$, we get $i(nj)\in IJ$, hence 
\begin{center}$n(ij)=\delta(ni)j-\delta(in)j+i(nj)\in IJ$ and $(ij)n=(in)j\in
IJ$.\end{center} 
As $IJ\subseteq J\subseteq I$, the minimality of $I$ gives $IJ=0$ or $I=IJ\subseteq J$. So we may assume
$IJ=0$. By \eqref{eq2} and \eqref{eq15}, $JI$ is also an ideal of $I$; applying Step 1 to it, we may also assume
$I(JI)=0$.

\textit{Step 2: $JI\neq0$.} The annihilator $\mathrm{Ann}_I=\{x\in I\mid xI=Ix=0\}$ is an ideal of $N$: for
$x\in\mathrm{Ann}_I$ we have 
\begin{center}
$(xn)i=(xi)n=0$, $(nx)i=(ni)x=0$, $i(xn)=\delta(ix)n-\delta(xi)n+x(in)=0$
and $i(nx)=\delta(in)x-\delta(ni)x+n(ix)=0$.
\end{center} Since $I^2=I\neq0$, we have $\mathrm{Ann}_I\neq I$, so $\mathrm{Ann}_I=0$. If
$JI=0$, then $J\subseteq\mathrm{Ann}_I=0$, a contradiction.

\textit{Step 3: $JI$ is an ideal of $N$.} By \eqref{eq15}, 
\begin{center}
$0=a(jb)=\delta(aj)b-\delta(ja)b+j(ab)=-\delta(ja)b+j(ab)$,
i.e. $j(ab)=\delta(ja)b$.
\end{center}
If $\delta=0$, then $JI=JI^2=0$, which contradicts Step 2; hence $\delta\neq0$. Put
$W=((ja)b)n$. On the one hand, $(jn)(ab)=(j(ab))n=\delta W$. On the other hand, by \eqref{eq15}, \eqref{eq2},
$aj=0$ and $a(jb)=0$,
\begin{gather*}
(jn)(ab)=\delta((jn)a)b-\delta(a(jn))b+a((jn)b),\qquad ((jn)a)b=((ja)n)b=W,\\
(a(jn))b=\big(-\delta(ja)n+j(an)\big)b=-\delta W+(j(an))b,\\
a((jn)b)=a((jb)n)=-\delta((jb)a)n+(jb)(an)=-\delta W+(jb)(an).
\end{gather*}
Hence $\delta(1-\delta)W=-\delta(j(an))b+(jb)(an)\in JI$, so $W\in JI$ and $(jn)(ab)=\delta W\in JI$. Thus
$(ji)n=(jn)i\in JI$, and $n(ji)=\delta(nj)i-\delta(jn)i+j(ni)\in JI$, because $(nj)i=(ni)j\in IJ=0$.
So $JI$ is a nonzero ideal of $N$ contained in $J\subseteq I$; therefore $JI=I$ and $J=I$.
\end{proof}

Define $N^{(0)}=N$ and define $N^{(n)}=N^{(n-1)}N^{(n-1)}$ for every integer $n\geq1$. Then $N$ is
called \textit{solvable} of index $n$ if $n$ is the minimal integer such that $N^{(n)}=0$. Define $N_L^1=N$
and $N_L^n=N_L^{n-1}N$. Then $N$ is called \textit{right nilpotent} of index $n$ if $n$ is the minimal integer
such that $N_L^n=0$. \textit{Left nilpotency} is defined similarly. Finally, we define $N^1=N$ and
\[N^n=\sum_{1\leq i\leq n-1}N^iN^{n-i},\quad\ n\geq2.\]
Then $N$ is called \textit{nilpotent} if $N^n=0$ for some positive integer $n$.

\begin{lemma}\label{lem50}
If $N$ is right nilpotent, then $N^2$ is nilpotent, in particular, $N$ is solvable.
\end{lemma}

\begin{proof}
The proof is the same as in \cite{b44}. In more detail (for every $\delta$): by \eqref{eq2} and induction
on $k$, $N_L^kN_L^2\subseteq N_L^{k+1}$ (write $u=u'v$, then $u(yw)=(u'(yw))v$); by the identity
\begin{center}
$a(xy)=\delta(ax)y-\delta(xa)y+x(ay)$ (proof of Lemma \ref{lem7}) \end{center}
and induction, $NN_L^k\subseteq N_L^k$; then by
\eqref{eq2} and induction $N_L^aN_L^b\subseteq N_L^{a+b-1}$. Hence $(N^2)^n\subseteq N_L^{n+1}$, so $N^2$ is nilpotent if $N$ is right
nilpotent. Finally $N^{(n+1)}\subseteq(N^2)^{2^n}$, and $(N^2)^{m+1}\subseteq(N^2)^m$, so $N$ is solvable if $N^2$ is nilpotent.
\end{proof}

\begin{lemma}\label{lem51}
For all $m\geq0$ and $n\geq1$, we have $(N^{(m)})_L^{3^n}\subseteq N^{(m+n)}$, provided that $\delta\neq0$.
\end{lemma}

\begin{proof}
The proof is the same as in \cite{b44}. In more detail: let \begin{center}
$M=N^{(m)}$ and $M'=MM=
N^{(m+1)}$;
\end{center} both are ideals by Lemma \ref{lem7}. For $a\in M'$, $b,c\in M$, by \eqref{eq1}, 
\begin{center}$\delta(ab)c=a(bc)+\delta(bc)a-
b(ac)\in MM'$ (we used $(ba)c=(bc)a$).
\end{center} Since $\delta\neq0$, $(M'M)M\subseteq MM'$, and by \eqref{eq2} $M_L^q\subseteq
M_L^{q-3}M'$ for $q\geq4$. By induction on $k$, $M_L^{3k}\subseteq(M')_L^k$, and then by induction on $n$, $(N^{(m)})_L^{3^n}\subseteq
(N^{(m+1)})_L^{3^{n-1}}\subseteq\ldots\subseteq N^{(m+n)}$.
\end{proof}

\begin{theorem}\label{thm52}
Let $N$ be a left $\delta$-Novikov algebra with $\delta\neq0$. Then the following are
equivalent:
\begin{enumerate}
\renewcommand{\labelenumi}{(\roman{enumi})}
\item[${(\rm i)}$] $N$ is right nilpotent.
\item[${(\rm ii)}$] $N^2$ is nilpotent.
\item[${(\rm iii)}$] $N$ is solvable.
\end{enumerate}
\end{theorem}

\begin{proof}
By Lemma \ref{lem50}, we have (i) $\Rightarrow$ (ii) $\Rightarrow$ (iii). To show (iii) $\Rightarrow$ (i), assume $N^{(n)}=0$ for a
positive integer $n$. By Lemma \ref{lem51}, we have $(N)_L^{3^n}=(N^{(0)})_L^{3^n}\subseteq N^{(0+n)}=0$.
\end{proof}

\begin{remark}
For $\delta=0$, Lemma \ref{lem51} and the implication 
$({\rm iii})\Rightarrow({\rm i})$ are false: the bicommutative
\textup($0$-Novikov\textup) algebra $\B N_{03}$: $e_1e_2=e_1$ satisfies $N^2=\langle e_1\rangle$, $(N^2)^2=0$, $N^{(2)}=0$,
but 
\begin{center}$e_1=(\ldots(e_1e_2)\ldots)e_2\in N_L^k$ for all $k$, so $N$ is not right nilpotent.
\end{center}The implications
$({\rm i})\Rightarrow({\rm ii})\Rightarrow({\rm iii})$ 
hold for every $\delta$.
\end{remark}


\end{document}